\pdfoutput=1
\RequirePackage{silence}
\documentclass[11pt]{article}

\usepackage[margin= 2 cm,bottom=19mm]{geometry}

\usepackage{amsthm}
\usepackage{amsmath}
\usepackage{amssymb}
\usepackage{setspace}
\usepackage{mathtools}
\usepackage{verbatim}
\usepackage{graphicx}
\usepackage[hidelinks]{hyperref}
\usepackage{cleveref}
\usepackage{graphicx}
\usepackage{enumitem}
\usepackage{framed}
\usepackage{subcaption}
\usepackage{floatrow}
\usepackage[T1]{fontenc}
\usepackage{tikz}
\usepackage{mathdots}
\usepackage{xcolor}
\usepackage{diagbox}
\usepackage{colortbl}
\usepackage[absolute,overlay]{textpos}
\usetikzlibrary{calc}
\usetikzlibrary{decorations.pathreplacing}
\usetikzlibrary{positioning,patterns}
\usetikzlibrary{arrows,shapes,positioning}
\usetikzlibrary{decorations.markings}

\tikzstyle{edge}=[very thick]
\definecolor{bostonuniversityred}{rgb}{0.8, 0.0, 0.0}
\definecolor{arsenic}{rgb}{0.23, 0.27, 0.29}
\tikzstyle{diredge}=[postaction={decorate,decoration={markings,
		mark=at position .95 with {\arrow[scale = 1]{stealth};}}}]

\newcommand{\fitellipsis}[2] 
{\draw [fill=green]let \p1=(#1), \p2=(#2), \n1={atan2(\y2-\y1,\x2-\x1)}, \n2={veclen(\y2-\y1,\x2-\x1)}
    in ($ (\p1)!0.5!(\p2) $) ellipse [ x radius=\n2/2+0cm, y radius=0.1cm, rotate=\n1];
}

\theoremstyle{plain}

\newtheorem*{thm*}{Theorem}
\newtheorem{thm}{Theorem}
\Crefname{thm}{Theorem}{Theorems}
\numberwithin{thm}{section}

\newtheorem*{lem*}{Lemma}
\newtheorem{lem}[thm]{Lemma}
\Crefname{lem}{Lemma}{Lemmas}

\newtheorem*{claim*}{Claim}
\newtheorem{claim}{Claim}
\crefname{claim}{Claim}{Claims}
\Crefname{claim}{Claim}{Claims}

\Crefname{prop}{Proposition}{Propositions}

\crefname{cor}{Corollary}{Corollaries}

\newtheorem{conj}[thm]{Conjecture}
\crefname{conj}{Conjecture}{Conjectures}

\newtheorem{qn}[thm]{Question}
\Crefname{qn}{Question}{Questions}

\Crefname{obs}{Observation}{Observations}

\newtheorem{ex}[thm]{Example}
\Crefname{ex}{Example}{Examples}

\theoremstyle{definition}

\Crefname{prob}{Problem}{Problems}

\Crefname{defn}{Definition}{Definitions}

\newtheorem*{defn*}{Definition}

\theoremstyle{remark}

\newcommand{\ceil}[1]{
    \left\lceil #1 \right\rceil
}
\newcommand{\floor}[1]{
    \left\lfloor #1 \right\rfloor
}

\expandafter\def\expandafter\normalsize\expandafter{%
    \normalsize
    \setlength\abovedisplayskip{5pt}
    \setlength\belowdisplayskip{5pt}
    \setlength\abovedisplayshortskip{5pt}
    \setlength\belowdisplayshortskip{5pt}
}

\usepackage[square,sort,comma,numbers]{natbib}
\setlist[itemize]{leftmargin=*}

\newcommand{\HH}{\mathcal{H}}
\newcommand{\cF}{\mathcal{F}}

\newcommand{\un}{\mathrm{un}}
\renewcommand{\ex}{\mathrm{ex}}
\newcommand{\eps}{\varepsilon}
\newcommand{\sun}{f}
\newcommand{\St}{\mathit{St}}
\DeclareFontFamily{OT1}{pzc}{}
\DeclareFontShape{OT1}{pzc}{m}{it}{<-> s * [1.10] pzcmi7t}{}
\DeclareMathAlphabet{\mathpzc}{OT1}{pzc}{m}{it}
\newcommand{\Su}{\mathit{Sf}}
\title{\vspace{-0.8cm} Unavoidable hypergraphs}

\author{
Matija Buci\'c\thanks{Department of Mathematics, ETH, Z\"urich, Switzerland. Email: \href{mailto:matija.bucic@math.ethz.ch} {\nolinkurl{matija.bucic@math.ethz.ch}}.}
 \and
 Nemanja Dragani\'c \thanks{Department of Mathematics, ETH, Z\"urich, Switzerland. Email: \href{mailto:matija.bucic@math.ethz.ch} {\nolinkurl{nemanja.draganic@math.ethz.ch}}.}
\and
Benny Sudakov\thanks{Department of Mathematics, ETH, Z\"urich, Switzerland. Email:
\href{mailto:benjamin.sudakov@math.ethz.ch} {\nolinkurl{benjamin.sudakov@math.ethz.ch}}.
Research supported in part by SNSF grant 200021\_196965.}
\and
Tuan Tran\thanks{Discrete Mathematics Group, Institute for Basic Science (IBS), Daejeon, Republic of Korea. Email:
\href{mailto:tuantran@ibs.re.kr} {\nolinkurl{tuantran@ibs.re.kr}}.
This work was supported by the Institute for Basic Science (IBS-R029-Y1).}
}
 \date{}

\begin{document}

\maketitle

\begin{abstract}
The following very natural problem was raised by Chung and Erd\H{o}s in the early 80's and has since been repeated a number of times. What is the minimum of the Tur\'an number $\ex(n,\HH)$ among all $r$-graphs $\HH$ with a fixed number of edges? Their actual focus was on an equivalent and perhaps even more natural question which asks what is the largest size of an $r$-graph that can not be avoided in any $r$-graph on $n$ vertices and $e$ edges?

In the original paper they resolve this question asymptotically for graphs, for most of the range of $e$. In a follow-up work Chung and Erd\H{o}s resolve the $3$-uniform case and raise the $4$-uniform case as the natural next step. In this paper we make first progress on this problem in over 40 years by asymptotically resolving the $4$-uniform case which gives us some indication on how the answer should behave in general.
\end{abstract} 

\section{Introduction}
The Tur\'an number $\ex(n,\HH)$ of an $r$-graph $\HH$ is the maximum number of edges in an $r$-graph on $n$ vertices which does not contain a copy of $\cF$ as a subhypergraph. For ordinary graphs (the case $r=2$), a rich theory has been developed (see \cite{FS-survey}), initiated by the classical Tur\'an's theorem \cite{Turan41} dating back to 1941. The problem of finding the numbers $\ex(n,\HH)$ when $r>2$ is notoriously difficult, and exact results are very rare (see surveys \cite{Furedi-survey, Keevash-survey, Sidorenko-survey, Benny-icm} and references therein).

The following very natural extremal question was raised by Chung and Erd\H{o}s \cite{chung1983unavoidable} almost 40 years ago. What is the minimum possible value of $\ex(n,\HH)$ among $r$-graphs $\HH$ with a fixed number of edges? The focus of Chung and Erd\H{o}s was on the equivalent inverse question which is perhaps even more natural. Namely, what is the largest size of an $r$-graph that we can not avoid in any $r$-graph on $n$ vertices and $e$ edges? This question was repeated multiple times over the years: it featured in a survey on Tur\'an-type problems \cite{Furedi-survey}, in an Erd\H{o}s open problem collection \cite{chung1997open} and more recently in an open problem collection from AIM Workshop on Hypergraph Tur\'an problems \cite{MPS-report}.

Following Chung and Erd\H{o}s we call an $r$-graph $\HH$ as above $(n,e)$-{\em unavoidable}, so if every $r$-graph on $n$ vertices and $e$ edges contains a copy of $\HH$. Their question now becomes to determine the maximum possible number of edges in an $(n,e)$-unavoidable $r$-graph. Let us denote the answer by $\un_r(n,e)$. In the graph case, Chung and Erd\H{o}s determined $\un_2(n,e)$ up to a multiplicative factor for essentially the whole range.
In a follow-up paper from 1987, Chung and Erd\H{o}s \cite{chung1987unavoidable} studied the $3$-uniform case and identified the order of magnitude of $\un_3(n,e)$ for essentially the whole range of $e$.\footnote{Their argument unfortunately contains an error: the proof of \cite[Lemma 6]{chung1987unavoidable} is incorrect. We fill this gap in \Cref{sec:Turan-bookstars}.} In the same paper Chung and Erd\H{o}s raise the $4$-uniform case as the natural next step since the $3$-uniform result fails to give a clear indication on how the answer should behave in general. In the present paper we resolve this question by determining $\un_4(n,e)$ up to a multiplicative factor for essentially the whole range of $e$.

\begin{thm} \label{thm:main}
 The following statements hold.
\begin{enumerate} 
    \item[\rm (i)] For $1\le e\leq n^2$, we have $\un_4(n,e)\approx 1$.
    \item[\rm (ii)] For $n^2\leq e \leq n^3$, we have $\un_4(n,e) \approx \min \big\{(e/n^2)^{3/4},(e/n)^{1/3}\big\}$.
    \item[\rm (iii)] For $n^3<e \ll \binom{n}{4}$, we have $\un_4(n,e)\approx \min\Big\{e^{4/3}/n^{10/3},\frac{e^{1/4}\log n}{\log \left(\binom{n}{4}/e\right)}\Big\}$.
\end{enumerate}
\end{thm}


The optimal unavoidable hypergraphs, or in other words hypergraphs which minimise the Tur\'an number, turn out to be certain combinations of sunflowers of different types. For this reason, it is essential for our proof of \Cref{thm:main} to have a good understanding of the Tur\'an numbers of sunflowers for a wide range of parameters. This turns out to be a well-studied problem in its own right.

\subsection{Sunflowers}
A family $A_1,\ldots,A_k$ of distinct sets is said to be a {\em sunflower} if there exists a {\em kernel} $C$ contained in each of the $A_i$ such that the {\em petals} $A_i\setminus C$ are disjoint. The original term for this concept was ``$\Delta$-system''. The more recent term ``sunflower'' coined by Deza and Frankl \cite{deza-frankl} has recently become more prevalent. For $r,k\ge 1$, let $\sun_r(k)$ denote the smallest natural number with the property that any family of $\sun_r(k)$ sets of size $r$ contains an ($r$-uniform) sunflower with $k$ petals. The celebrated Erd\H{o}s-Rado theorem \cite{ER60} from 1960 asserts that $\sun_r(k)$ is finite; in fact Erd\H{o}s and Rado gave the following bounds:
\begin{equation}\label{eq:Erdos-Rado bound}
(k-1)^r \le \sun_r(k) \le (k-1)^r r!+1.
\end{equation}

They conjectured that for a fixed $k$ the upper bound can be improved to $\sun_r(k) \le O(k)^r$. Despite significant efforts, a solution to this conjecture remains elusive. The current record is
$\sun_r(k) \le O( k \log(kr))^r$, established in 2019 by Rao \cite{Rao20}, building upon a breakthrough of Alweiss, Lovett, Wu and Zhang \cite{ALWZ19}.

Some 43 years ago, Duke and Erd\H{o}s \cite{DE77} initiated the systematic investigation of a closely related problem. Denote by $\Su_r(t,k)$ the $r$-uniform sunflower with $k$ petals, and kernel of size $t$. Duke and Erd\H{o}s asked for the Tur\'an number of $\Su_r(t,k)$. 
Over the years this problem has been reiterated several times \cite{Furedi-survey,chung1997open} including in a recent collaborative ``polymath'' project \cite{Polymath}. 
The case $k=2$ of the problem has received
considerable attention \cite{FF85, FR87, FW81, KL17, KL20, Sos73}, partly due to its huge impact in discrete geometry \cite{FR90}, communication complexity \cite{Sgall99} and quantum computing \cite{BCW98}.
Another case that has a rich history \cite{Erdos65, EG61, EKR61, Frankl13, Frankl17, Frankl17-newrange, FK18} is $t=0$ (a matching of size $k$ is forbidden); the optimal construction in this case is predicted by the Erd\H{o}s Matching Conjecture. 

For fixed $r,t$ and $k$ with $1 \le t \le r-1$ and $k \ge 3$ Frankl and F{\"u}redi \cite[Conjecture 2.6]{FF87} give a conjecture for the correct value of $\ex(n,\Su_r(t,k))$ up to lower order terms, based on two natural candidates for near-optimal $\Su_r(t,k)$-free $r$-graphs. They verify their conjecture for $r \ge 2t+3$, but otherwise, with the exception of a few particular small cases, it remains open in general. If we are only interested in asymptotic results the answer of $\ex(n,\Su_r(t,k))\approx n^{\max\{r-t-1,t\}}$ was determined by Frankl and F{\"u}redi \cite{FF85} and F{\"u}redi \cite{Furedi83}. 

Another natural question is what happens if we want to find large sunflowers, in other words if we only fix the uniformity $r$ and ``type'' of the sunflower, determined by its kernel size $t$, while allowing $k$ to grow with $n$. Further motivation for this question is that it is easy to imagine that it could be very useful to know how big a sunflower of a fixed type we are guaranteed to be able to find in an $r$ graph with $n$ vertices and $e$-edges. In particular, it is precisely the type of statement we require when studying the unavoidability problem of Chung and Erd\H{o}s.
In the graph case $r=2$ the question simply asks for the Tur\'an number of a (big) star and the answer is easily seen to be $\ex(n,\Su_2(1,k))\approx nk$. 
In contrast, the $3$-uniform case is already non-trivial: Duke and Erd\H{o}s \cite{DE77} and Frankl \cite{Frankl78} showed $\ex(n,\Su_3(1,k)) \approx nk^2$ while $\ex(n,\Su_3(2,k))\approx n^2k$. Chung \cite{chung1983unavoidableS} even managed to determine the answer in the $3$-uniform case up to lower order terms, while Chung and Frankl \cite{chung-frankl} determined $\ex(n,\Su_3(1,k))$ precisely for large enough $n$. Chung and Erd\H{o}s \cite{chung1987unavoidable} wrote in their paper that results for such large sunflowers with uniformity higher than $3$ are far from satisfactory. Here we make first progress in this direction, by solving asymptotically the $4$-uniform case.

\begin{thm}\label{thm:4-uniform-sunflowers}
For $2 \le k\le n$ we have
\begin{enumerate}
\item[\rm (i)]  $\ex(n,\Su_4(1,k))\approx k^2n^2$,
\item[\rm (ii)] $\ex(n,\Su_4(2,k))\approx k^2n^2$ and
\item[\rm (iii)] $\ex(n,\Su_4(3,k))\approx kn^3$.
\end{enumerate}
\end{thm}


\subsection{General proof strategy}
Our proof strategy for determining $f_r(n,e)$ for most of the range is as follows. In order to show an upper bound $f_r(n,e) \le D$ we need to show there is no $r$-graph with more than $D$ edges which is contained in every $r$ graph with $n$ vertices and $e$ edges. With this in mind we consider a number of, usually very structured, $n$-vertex $r$-graphs on $e$ or more edges, and argue they can not have a common subhypergraph with more than $D$ edges. The hypergraphs we use are often based on Steiner systems or modifications thereof. A major benefit of this approach is that our collection of hypergraphs often imposes major structural restrictions on possible common graphs which have close to $D$ edges as well and tells us where to look for our optimal examples of unavoidable hypergraphs which we need in order to show matching lower bounds, by upper bounding their Tur\'an numbers. 


\vspace{0.2cm}
\textbf{Organisation.}
In the following section we establish some preliminary results we will need later. In \Cref{sec:sunflower} we prove \Cref{thm:4-uniform-sunflowers}. In \Cref{sec:sparse-regime} we prove the first two parts of \Cref{thm:main}. In \Cref{sec:Turan-bookstars} we deal with the remaining regime. This section is split into several parts, in \Cref{subsec-3-uniform-gen-stars} we establish a number of $3$-uniform results we will need for the lower bounds     , which is proved in \Cref{subsec:-lower-bound}. We prove the upper bounds in \Cref{subsec:upper-bounds}. Finally, in \Cref{sec:conc-remarks} we make some final remarks and give a number of open problems and conjectures.

\vspace{0.2cm}
\textbf{Notation.}
A {\em generalised star} is defined recursively as follows: $\St_2(d)$ is the usual star $S_d$ with $d$ leaves, and $\St_r(d_1,\ldots,d_{r-1})$ is the $r$-graph in which all edges have a vertex $v$ in common and upon removal of $v$ from every edge we obtain $d_1$ copies of $\St_{r-1}(d_2,\ldots,d_{r-1})$.

Let $G$ be an $r$-graph, and let $S\subseteq V(G)$ such that $1\leq |S|\leq r-1$. Then the \emph{link graph}, denoted $L_S$, is the $(r-|S|)$-graph on $V(G)$, whose edges are the sets $T$ of size $r-|S|$ such that $S\cup T\in E(G)$.
The \emph{codegree} of $S$ in $G$ is defined as the number of edges of $G$ which contain $S$. 
If the codegree of $S$ is at least $k$, we say that $S$ is $k$-\emph{expanding}. We will refer to the immediate fact that in any $k$-uniform hypergraph the number of edges times the uniformity equals the sum of degrees over all vertices as the \textit{handshaking lemma}.

For non-negative functions $f$ and $g$ we write either $f \lesssim g$ or $f=O(g)$ to mean there is a constant $C>0$ such that $f(n)\le Cg(n)$ for all $n$, we write $f \gtrsim g$ or $f=\Omega(g)$ to mean there is a constant $c>0$ such that $f(n)\ge cg(n)$ for all $n$, we write $f\approx g$ to mean that $f\lesssim g$ and
 $f \gtrsim g$. To simplify the presentation we write $f\gg g$ or $g\ll f$ to mean that $f\ge Cg$ for a sufficiently large constant $C$\footnote{Note here that we are defining $\gg$, in a way which is more common in fields outside of combinatorics, namely $f \gg g$ does not mean $g=o(f)$ but is more similar to $g=O(f)$ with the exception that we are allowed to choose the constant in the big $O$ as small as we like, as long as it remains fixed.}, which can be computed by analysing the argument. In particular, in this paper choosing $C=2^{30}$ would be sufficient for all our arguments. All asymptotics are as $n\rightarrow \infty$ unless specified otherwise. 

From now on whenever we say optimal unavoidable graph, we mean it has the largest number of edges up to a constant factor. Throughout the paper we omit floor and ceil signs whenever they are not crucial, for the sake of clarity of presentation and since they would only, possibly, impact the constant factors.

\section{Preliminaries}
In this section we collect several simple results, that we use later on. The next two results will provide us with building blocks for examples of hypergraphs which will be useful both for proving lower bounds on Tur\'an numbers of sunflowers needed for \Cref{thm:4-uniform-sunflowers} as well as to force structure when proving upper bounds in \Cref{thm:main}. We include proofs for completeness.

\begin{lem}[Partial Steiner Systems]\label{lem:partial-Steiner}
Let $k>t>0$ be fixed integers. For 
every $n$ sufficiently large, there exists a 
$k$-graph $S(t,k,n)$ on $n$ vertices such that every set of vertices of size $t$ is contained in at most one edge, and the number of edges of $S(t,k,n)$ is at least
$
\Omega( n^t).
$
\end{lem}
\begin{proof}
Let $X$ be an arbitrary $k$-subset of $[n]$. The number of $k$-sets which intersect $X$ in $i$ elements is $\binom{k}{i}\binom{n-k}{k-i}$. Thus the total number of $k$-sets which intersect $X$ in at least $t$ elements is $\sum_{i=t}^{k}\binom{k}{i}\binom{n-k}{k-i}$. It follows that there exists a $k$-graph $G$ on $[n]$ such that:
\begin{itemize}
\item Any two edges of $G$ intersect in at most $t-1$ elements;
\item $|E(G)| \ge \frac{\binom{n}{k}}{\sum_{i=t}^{k}\binom{k}{i}\binom{n-k}{k-i}} \gtrsim n^t$.
\end{itemize}
This completes the proof.
\end{proof}

The above result as stated requires $k$ to be fixed, however for certain applications we will want to relax this assumption. The following result is a special case where $t=2$ and we allow $k \le \sqrt{n/2}$. Here we say a hypergraph is linear if no two of its edges intersect in more than one vertex.

\begin{lem}\label{lem:linear}
For $2\le k \le \sqrt{n/2}$, there exists a
linear $k$-graph on $n$ vertices with at least $n^2/4k^2$ edges.
\end{lem}
\begin{proof}
By Chebyshev's theorem, there is a prime $p$ between $n/2k$ and $n/k$. Look at the affine plane $\mathbb{F}_p^2$, and consider its subset $V=\{(x,y)\in\mathbb{F}_p^2\mid 0\le x\le k-1,\hspace{1pt} y\in\mathbb{F}_p$\}. The vertex set of our hypergraph will be $V$. Note that $|V|<n$.
The edges are \emph{partial lines} $L_{(x,y)}$, defined as follows for each $(x,y)\in \mathbb{F}_p^2$:
$$
L_{(x,y)}= \{(0,x)+t(1,y)\mid 0 \le t \le k-1\}.
$$
Notice that for distinct pairs $(x_1,y_1)$ and $(x_2,y_2)$ the corresponding lines $L_{(x_1,y_1)}$ and $L_{(x_2,y_2)}$ intersect in at most one vertex, so our $k$-graph is linear, and has $p^2\geq n^2/4k^2$ edges.
\end{proof}

The following simple lemma will often come in useful.
\begin{lem}\label{lem:star-match}
Let $G$ be a graph with at least $2k\ell$ edges and with no star $S_k$. Then $G$ contains a matching of size $\ell$. 
\end{lem}
\begin{proof}
We find an $\ell$-matching $M$ in $G$ as follows. Let $v\in V(G)$ be a non-isolated vertex, and take an arbitrary edge $(u,v)$ incident with $v$ and put it in $M$. Now delete all edges incident to $u$ and $v$ from $G$ and repeat this procedure. If we found less than $\ell$ such edges, we deleted at most $2(\ell-1)k$ edges in $G$, so there is an edge left which we can add to $M$.
\end{proof}

The next auxiliary lemma is a generalisation of \cite[Lemma 5]{chung1983unavoidable} which will come in useful when looking at higher uniformities. We give a different proof, as it illustrates an idea which will be used a lot later on.

\begin{lem}\label{lem:graph-many-stars}
    Any graph with $n$ vertices and $e=6sn \ge 6kn$ edges contains at least $\min\{s,\sqrt{sn}/k\}$ vertex-disjoint copies of the star $S_k$.
\end{lem}

\begin{proof} Let $t=\min\{s,\sqrt{sn}/k\}$.
If there are at least $t$ vertices with degree at least $(t-1)(k+1)+k=t(k+1)-1$ then we can greedily find $t$ vertex-disjoint copies of $S_k$. This means that by removing at most $t\cdot n \le sn$ edges we get a graph with maximum degree less than $t(k+1)$. Now let us take a maximal collection of vertex-disjoint $S_k$'s. Unless we are done there are at most $(t-1)(k+1)$ vertices spanned by these stars, so in total they touch less than $t^2(k+1)^2 \le 4t^2k^2 \le 4sn$ edges. So upon removing them we are left with at least $sn \ge kn$ edges and can find another $S_k$.
\end{proof}

\section{Tur\'an numbers of sunflowers}
\label{sec:sunflower}

In this section we give the proof of \Cref{thm:4-uniform-sunflowers}. 

\subsection{3-uniform case}

We will need the following $3$-uniform results, which were already established by Duke and Erd\H{o}s \cite{DE77} and Frankl \cite{Frankl78}. We include our, somewhat simpler proofs, for completeness and to illustrate the ideas we will use in the $4$-uniform case. There are only two different types of $3$-uniform sunflowers, namely $\Su_3(1,k)$ and $\Su_3(2,k)$.

\begin{lem}\label{lem:3-uniform-star}
Let $2 \le k \ll n$ we have $\ex(n,\Su_3(1,k))\approx k^2n.$
\end{lem}

\begin{proof}
For the lower bound, we split $[n]$ into disjoint sets: $A$ of size $n-k\ge n/2$, and $B$ of size $k$. Let our $3$-graph consist of all edges with one vertex in $A$ and two vertices in $B$. This $3$-graph has $\Omega (k^2n)$ edges and is $\Su_3(1,k)$-free. Indeed, if we can find a copy of $\Su_3(1,k)$ each of its edges contains two vertices in $B$, one of which is not the common vertex, so it uses at least $k+1$ vertices of $B$, which has size $k$, a contradiction. This shows $\ex(n,\Su_3(1,k))= \Omega(k^2n)$.
 

For the upper bound, we will show that every $3$-graph $G$ with $4k^2n$ edges contains a copy of $\Su_3(1,k)$. Let $G$ be such a $3$-graph
and suppose towards a contradiction that it does not contain an $\Su_3(1,k)$. For each $v\in V(G)$ let $D_v$ denote the (2)-graph on $V$ whose edges are the $2k$-expanding pairs $Y$ such that $v \cup Y \in E(G)$. $D_v$ does not contain matchings and stars of size $k$; if $D_v$ contained a $k$-matching then $v$ and this matching would make an $\Su_3(1,k)$ in $G$; if $D_v$ contained a star of size $k$ then we can greedily extend each edge of the star by a new vertex to obtain an $\Su_3(1,k)$, since the edges are $2k$-expanding. Using \Cref{lem:star-match} this implies that $D_v$ can have at most $2k^2$ edges. The number of edges of $G$ containing a $2k$-expanding pair is upper bounded by $\sum_v |D_v|\le 2k^2n$, so if we delete all such edges we are left with a $3$-graph $G'$ with at least $2k^2n$ edges with no $2k$-expanding pairs of vertices. Now take a vertex $v$ with degree at least $3|E(G')|/n \ge 4k^2$ in $G'$; it cannot have a star of size $2k$ in its link graph, as then the pair $v$ and centre of the star would be $2k$-expanding. Using \Cref{lem:star-match} this means there must be a $k$-matching in its link graph, which together with $v$ forms an $\Su_3(1,k)$ in $G$, so we are done.
\end{proof}


We now proceed to the second type of sunflowers.

\begin{lem}\label{lem:S_3(2k)}
For $2 \le k \ll n$ we have $\ex(n,\Su_3(2,k))\approx kn^2.$
\end{lem}
\begin{proof}
To prove the lower bound, we consider the linear $3$-graph $S(2,3,n)$ on $[n]$ with $\Omega(n^2)$ edges, given by \Cref{lem:partial-Steiner}. Let $G$ be a union of $k-1$ random copies of $S(2,3,n)$, where each copy is obtained by randomly permuting the vertices of $S(2,3,n)$. 
Since each pair of vertices lies in at most one edge from each copy of $S(2,3,n)$, $G$ does not contain a copy of $\Su_3(2,k).$ A fixed triple is chosen with probability $\Omega (1/n)$ in a random copy of $S(2,3,n)$, independently between our $k-1$ copies. Thus the probability that a given triple is chosen in one of our $k-1$ copies is at least $\Omega(k/n)$ so the expected number of chosen triples is $\Omega(kn^2)$, giving $\ex(n,\Su_3(2,k))= \Omega(kn^2)$.

We now turn to the upper bound. Let $G$ be a $3$-graph with $kn^2$ edges. By averaging, there must exist a pair of vertices belonging to at least $k$ edges, which make a copy of $\Su_3(2,k)$ in $G$.
\end{proof}

\subsection{4-uniform case}
In this subsection we determine the behaviour of the Tur\'an number of $4$-uniform sunflowers, namely we prove \Cref{thm:4-uniform-sunflowers}. We begin with $\Su_4(1,k)$.

\begin{lem}\label{lem:4-uniform t=1}
For $2\le k \ll n$ we have $\ex(n,\Su_4(1,k))\approx k^2n^2.$
\end{lem}

\begin{proof}
 We consider the lower bound first. 
We split the $n$ vertices into disjoint sets $A$ of size $n-k \ge n/2$ and $B$ of size $k$. 
Let $G$ be the $4$-graph consisting of edges which have two vertices in each of $A$ and $B$, so in total $G$ has $\approx k^2n^2$ edges. Note that $G$ is $\Su_4(1,k)$-free. Indeed, if we can find a copy of $\Su_4(1,k)$ each of its edges contains two vertices in $B$, one of which is not the common vertex, so it uses at least $k+1$ vertices of $B$, which has size $k$, a contradiction.

For the upper bound, we will show that every $4$-graph $G$ with $e \gg k^2n^2$ edges contains a copy of $\Su_4(1,k)$. Let $G$ be such a $4$-graph and suppose it does not contain a copy of $\Su_4(1,k)$. For each $v\in V(G)$, let $D_v$ denote the set of $3k$-expanding triples $Y$ such that $v \cup Y \in E(G)$. So $D_v$ is a $3$-graph. If some $D_v$ has at least $e/(2n) \gg k^2n$ edges then by \Cref{lem:3-uniform-star} we can find an $\Su_3(1,k)$ in $D_v$ and greedily extend it to an $\Su_4(1,k)$ in $G$, so we may assume that each $D_v$ has at most $e/(2n)$ edges. The number of edges of $G$ containing a $3k$-expanding triple is upper bounded by $\sum_v |D_v|\le e/2$ so if we delete all such edges we are left with a $4$-graph $G'$ with at least $e/2$ edges and no $3k$-expanding triple of vertices. 

Now look at pairs of vertices which are $18k^2$-expanding. Any such pair has no star of size $3k$ in its link graph as that would give a $3k$-expanding triple, so by \Cref{lem:star-match} it has a matching of size $3k$. If we can find a star of size $k$ formed by the $18k^2$-expanding pairs, then using the matchings we found in the link graphs we can once again greedily extend it into a copy of $\Su_4(1,k)$. Hence, the total number of $18k^2$-expanding pairs is at most $kn$. They can lie in at most $3k^2n^2$ different edges (since the third vertex we can choose in $n$ many ways but the fourth in at most $3k$, because there are no $3k$-expanding triples). Deleting all such edges from $G'$ we obtain $G''$ with at least $e/4 \ge 18k^2n^2$ edges, without $18k^2$-expanding pairs, which is a contradiction (since by density $G''$ must have an $18k^2$-expanding pair). 
\end{proof}

\textbf{Remark.} By induction this easily extends to higher uniformities, giving $\ex(n,\Su_r(1,k))\approx_r k^2n^{r-2}.$ 


We now turn to the second type of $4$-uniform sunflowers, namely $\Su_4(2,k).$

\begin{lem}\label{lem:4u-ds-m}
For $2 \le k\ll n $ we have $\ex(n,\Su_4(2,k))\approx k^2n^2.$
\end{lem}
\begin{proof}

We prove a lower bound first.
Let us say that a ($2$-)graph $G$ is \emph{good} if it has at most $2k$ vertices and each of its edges is contained in a copy of $K_4$. If we can show that there exists 
$m$ edge disjoint copies of good graphs $G_1,\ldots,G_m$ on the same vertex set $[n]$, with at least $\Theta(n^2k^2)$ copies of $K_4$ in total among $G_1,\ldots, G_m$, then we would be done. Indeed, we can construct a $4$-graph $H$ on $[n]$ by putting a $4$-edge in $H$ for any $4$ vertices which induce a copy of $K_4$ in one of $G_1,\ldots, G_m$; this $4$-graph has at least $\Theta(n^2k^2)$ edges, and by assumption each vertex pair $P$ in $[n]$ is an edge of at most one graph $G_i$, and therefore all the $4$-edges in $H$ which contain $P$ contain only vertices from $G_i$ of which there are at most $2k$, so no pair can be the centre of a sunflower $\Su_4(2,k)$ which has $2k+2$ vertices.

Now we show the existence of such $G_1,\ldots,G_m$, for $m=\frac{n^2}{48k^2}$. We choose $2k$ vertices uniformly at random, with repetition from $[n]$ and choose $G_1$ to be the complete graph on these $2k$ vertices. Suppose we obtained graphs $G_1,\ldots, G_i$, where $i<m$. Choose again a set of uniformly random $2k$ vertices, and choose $G_{i+1}$ to be the complete graph on these vertices from which we remove all the edges in $G_1,\ldots, G_i$ and after this we remove all edges not participating in a $K_4$.

Notice that for each graph $G_{i}$, with $i\in [m]$, the expected number of $K_4$'s is at least $$\binom{2k}{4}\left(1-6i \cdot (2k/n)^2\right)\cdot \frac{1}{2}\ge \Theta(k^4),$$ 
since probability that $4$ randomly sampled vertices are different is at least $1-6/n\ge 1/2$ and by a union bound the probability that one of its $6$ edges already got chosen in some $G_1,\ldots, G_{i}$ is at most $6i \cdot (2k/n)^2 \le 1/2$.
So the total expected number of $K_4$'s among $G_1,\ldots, G_m$ is by linearity of expectation at least $m\cdot \Theta(k^4)=\Theta (n^2k^2)$ and we are done.

For the upper bound, there must be a vertex $v$ with degree $\gg k^2n$ in any graph on $\gg k^2n^2$ edges, so we can find an $\Su_3(1,k)$ in its link graph, by Lemma \ref{lem:3-uniform-star}, which together with $v$ forms a copy of $\Su_4(2,k)$.
\end{proof}

Finally, we deal with the third and last kind, namely $\Su_4(3,k).$ 

\begin{lem}
For $2 \le k\ll n$ we have $\ex(n,\Su_4(3,k))\approx kn^3.$
\end{lem}
\begin{proof}
 For the lower bound, we take a $4$-graph $G$ which is a union of $k-1$ random copies of a $4$-graph $S(3,4,n)$ on $n$ vertices with $\Omega(n^3)$ edges, given by \Cref{lem:partial-Steiner}. A single triple of vertices lies in at most one edge for each copy so in total in at most $k-1$ edges; this means there is no $\Su_4(3,k)$ in $G$. Each fixed quadruple is an edge of $G$ with probability $\Omega(1/n)$, independently between different choices. Thus the probability that a quadruple is chosen in $G$ is at least $\Omega (k/n)$, 
so the expected number of chosen triples is at least $\Omega(kn^3)$. Therefore, $\ex(n,\Su_4(3,k))\ge \Omega( kn^3).$ 

 Let $G$ be a $4$-graph on $n$ vertices with $kn^3$ edges. By the pigeonhole principle there is a triple of vertices belonging to at least $k$ edges, which makes an $\Su_4(3,k)$. This shows $\ex(n,\Su_4(3,k))\le kn^3.$
\end{proof}

The above three results establish \Cref{thm:4-uniform-sunflowers} when $k\ll n$. This shows that $\ex(n,\Su_4(i,k))\approx n^4$ for any $k = cn$ for some small enough constant $c$. The bound in the remaining range, namely when $k \ge cn$, is immediate since $\Su_4(i,k)\subseteq \Su_4(i,k')$ for any $k \le k'$ in the case of lower bounds and by the trivial bound $\ex(n,\Su_4(i,k))\le \binom{n}{4}$ in the case of upper bounds, since we are only interested in bounds up to constant factor.

\section{Unavoidability, sparse regimes} \label{sec:sparse-regime}
In this section we prove \Cref{thm:main} for $e \ll n^3$, so for the majority of the first two regimes. We note that in order to prove \Cref{thm:main} it is sufficient to prove it for the regimes $e \ll n^2, n^2 \ll e \ll n^3$ and $n^3\ll e \ll n^4$ since in the remaining cases $e\approx n^2, e \approx n^3$ the bounds of the regimes match (up to a constant factor) and $\un_4(n,e)$ is monotone in $e$. We begin with the sparsest regime $e \ll n^2,$ which is quite simple to handle but illustrates the general approach.


\begin{thm}\label{thm:main-reformulation-i}
    For $1\le e\ll n^2$, we have $\un_4(n,e)=1$.
\end{thm}

\begin{proof}
Starting with the upper bound, let $\mathcal{H}$ be an $(n,e)$-unavoidable graph where $e\le cn^2$ for some sufficiently small constant $c>0$. This means it is contained in any $4$-graph on $n$ vertices with at least $e$ edges and our task is to show that it must consist of only one edge. To see this, observe that the $4$-graph $S(2,4,n)$, given by \Cref{lem:partial-Steiner}, has $\Omega(n^2)\ge e$ edges so it must contain $\mathcal{H}$ as a subgraph. This forces $\mathcal{H}$ to be linear. Similarly, the $n$-vertex $4$-graph which consists of all edges which contain two fixed vertices has $\Omega(n^2) \ge e$ edges so also has $\mathcal{H}$ as a subgraph. This forces any two edges of $\mathcal{H}$ to intersect in at least $2$ vertices. Since $\mathcal{H}$ must also be linear this means it can have at most $1$ edge. 

The lower bound is immediate, since any $n$-vertex graph with $e$ edges contains an edge ($e \ge 1$) so a single edge graph is $(n,e)$-unavoidable showing $\un_4(n,e)\ge 1$.
\end{proof}

We now turn to the upper bound for the second regime.

\begin{thm}\label{thm:middle-upper-bound}
For $n^2\ll e \ll n^3$ we have $\un_4(n,e) \lesssim \min \{(e/n^2)^{3/4},(e/n)^{1/3}\}$.
\end{thm}

\begin{proof}
Let $H$ be an $(n,e)$-unavoidable $4$-graph and let $k=c\sqrt{e}/n$, for $c>0$ large enough. This means that any $n$-vertex $4$-graph with at least $\Omega(k^2n^2)$ edges must contain $H$ (by choosing $c$ large enough, since $e=n^2k^2/c^2$). Note that the regime bounds imply $1 \ll k \ll \sqrt{n}$. To show the bound it suffices to show that $|E(H)| \le \min \{2k^{3/2},(k^2n)^{1/3}\}.$ 
In order to do this, we will consider a number of examples of $4$-graphs with more than $e$ edges. Each of them will reveal some additional information on how $H$ should look like and allow us to conclude it can't have more than the claimed number of edges.
\begin{itemize}
    \item The $4$-graph $S(3,4,n)$, given by \Cref{lem:partial-Steiner}, has $\Omega(n^3) \ge e$ edges. So it forces $H$ to have no two edges intersecting in three vertices.
    \item The graph with one special vertex contained in all of $\binom{n-1}{3}$ possible edges implies that all edges of $H$ must contain a common vertex, say $v$. Let $H_3$ be the link graph of $v$ (so a $3$-graph).
    \item We take the graph obtained from an $n$-vertex linear $k$-graph by taking every 4-subset of every edge in this $k$-graph as an edge. Note that by \Cref{lem:linear}, since $k \ll \sqrt{n}$, we can find such a 4-graph with $\Omega(k^2n^2)$ edges. This implies that $H_3$ splits into components of size at most $k$ each, since $v$ is contained in all $4$-edges and by construction any two $4$-edges which intersect in more than one vertex belong to a single $k$-edge of our starting linear $k$-graph. 
    \item Let us take the hypergraph with sets of vertices $V_1$ of size $k$ and $V_2$ of size $n-k$ such that we choose any pair of vertices in $V_1$ and any pair of vertices in $V_2$ and join them in an edge. This 4-graph has $\Omega (k^2n^2)$ edges. Since every edge has at least two vertices in $V_1$ this means every edge of $H_3$ must have at least one vertex in $V_1$; in other words, $H_3$ has a cover $C$ of size at most $k$ and (in particular) $H_3$ has at most $k$ components.
    \item Let us take a set $S$ of $2k^2 \ll n$ vertices and split them into $k^2$ pairs. We join each of these $k^2$ pairs with every pair among $n-|S|$ vertices outside of $S$ into a $4$-edge. This gives us an $n$ vertex $4$-graph with $\Omega (k^2n^2)$ edges, so we must be able to find a copy of $H$ inside it. If $v$ is embedded inside $S$ and we let $w$ be its pair, then $v$ and $w$ belong to each edge of $H$. 
    We claim this implies that $H$ has at most $k$ edges. To see this observe first that every edge of $H_3$ contains $w$, so $H_3$ only has a single component. By the third point we know it consists of at most $k$ vertices. We further know, by the first point, that if we remove $w$ we get a matching, since otherwise we would have two edges of $H$ which intersect in $3$ vertices. This implies $H$ has at most $k$ edges and we are done. 
    So, $v$ must be embedded outside of $S$. If a vertex of $C$ is in $S$ then it participates in at most one edge of $H$ (since we know each such edge contains the vertex of $C$, its pair in $S$ and $v$ and there is only one edge of $H$ containing any triple of vertices), so such vertices contribute at most $k$ edges. If we remove these edges from $H$ we know that in the remaining $4$-graph upon removing $v$ and the vertex of $C$ from an edge we obtain one of our pairs in $S$. 
\end{itemize}
Since we removed at most $k\ll \min \{k^{3/2},(k^2n)^{1/3}\}$ edges, so at most a constant proportion of edges in $H$, we may assume we started with $H$ in which such edges did not exist.
Putting together the observations so far we know that $H$ has a fixed vertex $v$ in all edges, its link graph is the $3$-graph $H_3$ which consists of at most $k$ copies of a subhypergraph of $\Su_3(1,k)$, whose centres are vertices of $C$ and whose petals upon removal of the centre vertices give a matching $M$ (our pairing of vertices in $S$) which in total has size at most $k^2$. The following two further examples provide us with one of the desired bounds each.
    
\begin{itemize}
    \item Let $V_1$ induce $k$ disjoint copies of $K_{\sqrt{k},\sqrt{k}}$, we extend each edge of this graph into $4$-edges by adding every possible pair of the remaining vertices (the set of which we denote by $V_2$). Since $|V_1|=2k^{3/2} \ll n$ if we set $|V_2|=n-|V_1|$ this $4$-graph will have $n$ vertices and $\Omega(k^2n^2)$ edges, so contains $H$. Let us consider the edges of $M$ containing a vertex embedded in $V_1$. There can be at most $2k^{3/2}$ such edges since $|V_1|=2k^{3/2},$ so upon deleting all corresponding edges of $H$ we are left with a subgraph of $H$ in which $v$ and $C$ got embedded into $V_1$ (or we are left with an empty graph). But this implies 
    $|C|\le 2\sqrt{k}$, so again there are at most $2k^{3/2}$ edges of $H$ remaining (since we know that if we fix a vertex from $C$, in addition to $v$, their link graph is a matching of size at most $k$). It follows that $|E(H)|\le 4k^{3/2}$ giving us the first part of the result. 
    
    \item Take a set $V_1$ of $(k^2n)^{1/3} \ll n$ vertices and let $V_2$ be the set of remaining vertices. We make a $4$-graph by taking any triple in $V_1$ and a single vertex in $V_2$. This gives us $\Omega (k^{2}n^2)$ edges and implies $|E(H)|\le (k^2n)^{1/3}$ since among every pair of vertices in $M$ at least one must be in $V_1$.
\end{itemize} 
This completes the proof.
\end{proof}





The rest of this section is devoted to the upper bound part of the following theorem, as the lower bound follows from \Cref{thm:middle-upper-bound}. Analysing the above proof narrows down the possibilities for an optimal unavoidable graph significantly, leading us to $\St_4(\sqrt{k},k,1)$ as a natural candidate for an optimal unavoidable graph. This indeed turns out to be the case as a consequence of the following result.
\begin{thm}\label{thm:middle-lower-bound}
For $2 \le k \le n^{2/3}$ we have $\ex(n,\St_4(\sqrt{k},k,1))\approx \max\{k^2n^2,k^{9/2}n\}$.
\end{thm}

Before proving this result let's see why it gives the desired lower bound for the unavoidability problem. We want to show that there is an $(n,e)$-unavoidable $4$-graph with $\gtrsim \min \{(e/n^2)^{3/4},(e/n)^{1/3}\}$ edges, for any $n^2 \ll e \ll n^{3}$. To do this we choose $k$ as large as possible, so that $e \gg \max\{k^2n^2,k^{9/2}n\}$, which means that $k \gtrsim \min \{({e}/{n^2})^{1/2},({e}/{n})^{2/9}\}$. By our choice of $k$ the above theorem applies and tells us that any $4$-graph with $n$ vertices and $e$ edges contains $\St_4(\sqrt{k},k,1)$, i.e.\ it is $(n,e)$-unavoidable. This implies there is an $(n,e)$-unavoidable $4$-graph with $k^{3/2} \gtrsim \min \{({e}/{n^2})^{3/4},({e}/{n})^{1/3}\}$ edges, as desired. So, combining \Cref{thm:middle-lower-bound} and \Cref{thm:middle-upper-bound} we obtain the desired result for the middle range.

\begin{thm}\label{thm:main-reformulation-ii}
    For $n^2\ll e \ll n^3$, we have $\un_4(n,e) \approx \min \big\{({e}/{n^{2}})^{3/4},\left({e}/{n}\right)^{1/3}\big\}$.
\end{thm}


Let us now turn to the proof of \Cref{thm:middle-lower-bound}. For the upper bound our task is to show that any $4$-graph $G$ on $n$ vertices with $e\gg \max(k^2n^2,k^{9/2}n)$ edges contains a copy of $\St_4(\sqrt{k},k,1)$. To this end, note that by the pigeonhole principle, the link graph $L_v$ of some vertex $v\in V(G)$ must have $e/n\gg \max(k^2n,k^{9/2})$ triples. If $L_v$ contains $\sqrt{k}$ vertex-disjoint copies of $\Su_3(1,k)$, then we are done. Unfortunately, a $3$-graph on $n$ vertices with $\gg \max(k^2n,k^{9/2})$ edges may not have more than one vertex-disjoint copy $\Su_3(1,k)$, let alone $\sqrt{k}$ copies. For example, the $3$-graph consisting of all triples containing a fixed vertex has $\binom{n-1}{2}$ edges (which is large enough when $k \ll n^{4/9}$), and it clearly does not contain two disjoint copies of $\Su_3(1,k)$. As the next result shows, one can remedy the situation by imposing a boundedness condition on the codegrees.

\begin{lem}\label{lem-pairs-of-stars}
Let $k \ge 2$. Every $n$-vertex $3$-graph with at least $e \gg \max(k^2n,k^{9/2})$ edges, in which every pair of vertices has codegree at most $3k^{3/2}$, contains $\sqrt{k}$ vertex-disjoint copies of $\Su_3(1,k)$.
\end{lem}


\begin{proof}
Assume first that there exist vertices $v_1,\ldots, v_{\sqrt{k}}$ with degree at least $18k^3$.
Since there are no pairs of vertices with codegree larger than $3k^{3/2}$, we know that the link graph of any $v_i$ does not contain a star of size $3k^{3/2}$, so by \Cref{lem:star-match} it must contain a matching of size $3k^{3/2}$.
Now assume we have found $i-1$ disjoint copies of $\Su_3(1,k)$ centred at $v_1,\ldots,v_{i-1}$ and not using any other $v_j$'s. Let us consider a matching of size $3k^{3/2}$ in the link graph of $v_i$. At most 
$\sqrt{k}+2(i-1)k<2k^{3/2}$ of the pairs in the matching already contain a vertex from either $\{v_i,\ldots,v_{\sqrt{k}}\}$ or belonging to one of our $(i-1)$ already found copies of $\Su_3(1,k)$'s. The remaining pairs make a matching of size at least $k$ in the link graph of $v_i$ while avoiding $\{v_1,\ldots,v_{\sqrt{k}}\}$ as well as any already used vertex. This gives us a new $\Su_3(1,k)$ centred at $v_i$ which is disjoint from the previous ones. After repeating $\sqrt{k}$ many times we get the desired $\sqrt{k}$ vertex-disjoint copies of $\Su_3(1,k)$.

So we may assume there are less than $\sqrt{k}$ vertices with degree at least $18k^3$. Each such vertex belongs to at most $3k^{3/2}n$ edges (the second vertex we can choose in $n$ ways and then third in $3k^{3/2}$ ways since codegrees are at most $3k^{3/2}$). Thus in total there are at most $3k^2n$ edges containing a vertex with degree at least $18k^3$. We delete these edges and are left with at least $e/2$ edges.

Suppose we have found $i-1$ vertex-disjoint copies of $\Su_3(1,k)$ in the remaining $3$-graph, where $1 \le i\le \sqrt{k}$. There are at most $(i-1)(2k+1)\cdot 18k^3<54k^{9/2}$ edges which contains a vertex from one of these copies. Removing these edges, we are left with at least $e/4 \gg k^2n$ edges. Thus \Cref{lem:3-uniform-star} applies giving us a new $\Su_3(1,k)$. This way we obtain $\sqrt{k}$ vertex-disjoint copies of $\Su_3(1,k)$.
\end{proof}

We are now ready to finish our analysis of the middle range by proving \Cref{thm:middle-lower-bound}.

\begin{proof}[Proof of \Cref{thm:middle-lower-bound}]
The lower bound follows from \Cref{thm:middle-upper-bound}.


Turning to the upper bound, set $m \gg \max(k^2n,k^{9/2})$ and let $G$ be an $n$-vertex $4$-graph with at least $nm$ edges. Suppose to the contrary that $G$ has no copies of $\St_4(\sqrt{k},k,1)$. 

Let us first consider the case that there are at most $nm/2$ edges containing a $3k^{3/2}$-expanding triple in $G$. Remove all these edges to obtain a $4$-graph $G'$ with at least $nm/2$ edges in which there is no $3k^{3/2}$-expanding triple. By the handshaking lemma, we can find a vertex $x$ in $G'$ of degree at least $2m$. If we look at the $3$-graph $G_x$ obtained by removing $x$ from all these edges we know it has at least $2m$ edges, additionally we know that no pair of vertices is $3k^{3/2}$-expanding in $G_x$ as such pair together with $x$ would give a $3k^{3/2}$-expanding triple. So all codegrees in $G_x$ are at most $3k^{3/2}$ and \Cref{lem-pairs-of-stars} applies, giving us $\sqrt{k}$ vertex-disjoint copies of $\Su_3(1,k)$ which together with $x$ make a copy of $\St_4(\sqrt{k},k,1)$, a contradiction.

Therefore, there are at least $nm/2$ edges which contain a $3k^{3/2}$-expanding triple. For any $v \in V(G)$ let $D_v$ denote the $3$-graph consisting of triples $X$ such that $v \cup X \in E(G)$ and $X$ is $3k^{3/2}$-expanding. Note that we have $\sum_{v\in V(G)} |D_v| \ge nm/2$ since each edge of $G$, containing a $3k^{3/2}$-expanding triple, contributes at least $1$ to this sum. Hence, there exists a vertex $x$ with $|D_x| \ge m/2$.

\begin{claim}\label{claimA}
There is no star of size $\sqrt{k}$ consisting of $3k^{3/2}$-expanding pairs in the $3$-graph $D_x$.
\end{claim}
\begin{proof}
Suppose to the contrary that we can find distinct vertices $v_0,\ldots, v_{\sqrt{k}}$ such that $\{v_0,v_i\}$ makes a $3k^{3/2}$-expanding pair for all $1 \le i \le \sqrt{k}$. We know that $\{v_0,v_i\}$ completes into an edge of $D_x$ in at least $3k^{3/2}$ different ways, so we can greedily find $k$ vertices $v_{i1},\ldots v_{ik}$ such that $\{v_0,v_i,v_{ij}\}\in E(D_x)$ and all $v_i,v_{ij}$ are distinct, for all $1\leq i\leq \sqrt{k}$ and $1\leq j\leq k$. The last part is due to the fact that we choose at most $1+\sqrt{k}+k^{3/2}<2k^{3/2}$ vertices in total and each pair completed an edge in at least $3k^{3/2}$ ways so we always have an unused vertex to choose for our $v_{ij}$'s.

Now since each $\{v_0,v_i,v_{ij}\}\in E(D_x)$ and is in particular expanding, we can extend it into an edge of $G$ in $3k^{3/2}$ ways so again greedily we obtain $\St_4(\sqrt{k},k,1)$ as it contains $1+\sqrt{k}+2k^{3/2}< 3k^{3/2}$ vertices so we always have a new vertex to choose to extend $\{v_0,v_i,v_{ij}\}$. This is a contradiction.
\end{proof}

For $y \neq x$, let $D_{xy}$ denote the set of pairs $X$ such that $y \cup X \in E(D_x)$ and $X$ is $3k^{3/2}$-expanding in $D_x$.

\begin{claim}\label{claimB}
 For every $y \ne x$, one has $|D_{xy}|\le 6k^{3/2}n$. \end{claim}
\begin{proof}
Suppose to the contrary that $|D_{xy}|\ge 6k^{3/2}n$ for some $y\ne x$. According to \Cref{lem:graph-many-stars} (with $s=k^{3/2}$), the graph $D_{xy}$ contains $\frac{\sqrt{k^{3/2}n}}k \ge \sqrt{k}$ (using $n \ge k^{3/2}$) vertex-disjoint copies of $S_k$. Label the edges of these stars by $e_1,e_2,\ldots,e_{k^{3/2}}$. As $y\cup e_i$ is a $3k^{3/2}$-expanding triple for every $i$, we can greedily find distinct vertices $v_1,\ldots,v_{k^{3/2}} \in V(G)\setminus \bigcup_{j}e_j$ such that $y\cup e_{1}\cup v_{1},\ldots, y\cup e_{k^{3/2}}\cup v_{k^{3/2}}$ are edges of $G$. This yields a copy of $\St_4(\sqrt{k},k,1)$, a contradiction.
\end{proof}

\begin{claim}\label{claimC}
There are at most $\sqrt{k}$ vertices $y$ having $|D_{xy}| \ge 6k^2.$ 
\end{claim}
\begin{proof}
Suppose otherwise and let $v_1,\ldots, v_{\sqrt{k}}$ denote vertices with $|D_{xv_i}|\ge 6k^2$.
Since by \Cref{claimA} there are no stars of size $\sqrt{k}$ in $D_{xv_i}$ we know by \Cref{lem:star-match} that there must be a matching of size $3k^{3/2}$ in $D_{xv_i}$. 
Now assume for some $i\le \sqrt{k}$ we have found $i-1$ vertex-disjoint copies of $\Su_3(1,k)$ centred at $v_1,\ldots,v_{i-1}$ which do not use vertices from $\{v_i,\ldots,v_{\sqrt{k}}\}$. We know there is a matching of size $3k^{3/2}$ in $D_{xv_i},$ at most $\sqrt{k}+2k^{3/2}$ of the pairs in the matching use a vertex which already belongs to one of our $\Su_3(1,k)$'s so we can still find a $k$-matching in $D_{xv_i},$ avoiding any already used vertices. This gives us a new $\Su_3(1,k)$ centred at $v_i$ disjoint from the previous ones and we may repeat this $\sqrt{k}$ many times. Finally, these copies of $\Su_3(1,k)$ when joined with $x$ make a copy of $\St_4(\sqrt{k},k,1)$ which is a contradiction. 
\end{proof}


Finally, it follows from \Cref{claimB,claimC} that the number of edges in $D_x$ containing a $3k^{3/2}$-expanding pair is at most $\sum_{y} |D_{xy}|\le \sqrt{k}\cdot 6k^{3/2}n+n\cdot 6k^2\le m/4$.
Deleting all such edges we are left with $m/4$ edges of $D_x$ such that no pair of vertices is $3k^{3/2}$-expanding. \Cref{lem-pairs-of-stars} implies we can find $\sqrt{k}$ disjoint copies of $\Su_3(1,k)$ which together with $x$ gives us a $\St_4(\sqrt{k},k,1),$ a contradiction. This completes the proof.
\end{proof}




\section{Unavoidability, the dense regime} \label{sec:Turan-bookstars}
In this section we deal with the last regime, when $e \gg n^3$. Note once again that by monotonicity of $\un_4(n,e)$ and combining it with \Cref{thm:main-reformulation-ii} this will imply the result for $e \approx n^3$ as well.


For the majority of the regime the optimal unavoidable $4$-graphs turn out to be based on generalised stars.
Specifically, for $e=kn^3$ they will be disjoint unions of $\St_4((n/k)^{1/3},(n/k)^{1/3},k)$ of suitable size. Unfortunately, in this regime upper bounds do not force as much structure as they do in the previous section, so we begin with the lower bounds. With this in mind our first goal is to determine $\ex(n,\St_4(d_1,d_2,d_3))$ for  $d_1=d_2=(n/k)^{1/3}$ and $d_3=k$ (although our methods should allow one to answer this question in general as well). 

In order to find a copy of $\St_4(d_1,d_2,d_3)$ in a graph we will either find $d_1$ copies of $\St_3(d_2,d_3)$ inside the link graph of a vertex or find $\St_3(d_1,d_2)$ consisting of so called expanding triples, namely triples of vertices which belong to many edges of our graph. In the former case we are done immediately, in the latter we can use the expansion of the triples to greedily extend each edge of our $\St_3(d_1,d_2)$ using $d_3$ new vertices. To find disjoint copies of $\St_3(d_2,d_3)$ or $\St_4(d_1,d_2,d_3)$ we can simply remove any already used vertex from the graph and argue that the remainder still contains enough edges to find a new copy. Unfortunately, as we have already seen in the previous section this approach fails to provide enough stars in most cases. A way around this is to embed the leaves of all our stars (since there are many of them) among vertices with low degrees, the 3rd layer vertices among vertices with only slightly higher degree and so on.
This approach requires very good understanding of Tur\'an numbers of $3$-uniform generalised stars and their unions which we give in the following section. 

In the subsequent section we show the desired bounds for our $4$-uniform case. Interestingly, towards the end of the range, as $e$ approaches $\binom{n}{4}$, generalised stars stop being the optimal examples and are replaced with (disjoint unions) of complete $4$-partite graphs, which we will find through a combination of the K\"{o}v\'ari-S\'os-Tur\'an theorem and a similar embedding trick where we embed largest parts of our already found $r$-partite graphs into vertices with low degree, in order to be able to find many disjoint copies.

\subsection{Tur\'an numbers of 3-uniform generalised stars}\label{subsec-3-uniform-gen-stars}
In this section we will give a number of upper bounds on Tur\'an numbers of generalised stars, and their disjoint unions. The first two results determine these Tur\'an numbers up to constant factors and we prove them in full generality since we find them interesting in their own right. The subsequent three lemmas allow us to do even better (find our stars in graphs with even less edges) if we know certain additional properties of our graph or give us more control where in the graph we can find our stars. We do not state these in full generality, but rather for the cases which arise naturally in the proof of our $4$-uniform result.

The following lemma, which generalises \Cref{lem:3-uniform-star}, determines the Tur\'an number of any $3$-uniform generalised star, up to a constant factor.

\begin{thm}\label{prop:3-uniform-single-star-lastr}
    For every positive integers $n,h$ and $k$, we have 
    \[
    \ex(n,\St_3(h,k))\lesssim \max\{kn^2,h^2k^2n\}.
    \]
\end{thm}
 
\begin{proof}

Let $G$ be a $3$-graph with $e \gg \max(kn^2,h^2k^2n)$ edges. We want to find $h$ vertex-disjoint copies of $S_k$ in the link graph of some vertex. Let $D_v$ be the subgraph of the link graph of $v$ consisting of pairs $Y$ for which $v \cup Y$ is an edge of $G$ and $Y$ is $3hk$-expanding.

If at least $e/2$ edges in $G$ contain a pair which is $3hk$-expanding then there exists a $D_v$ of size at least $e/(2n)$, as $\sum_{v\in V}|D_v|\geq e/2$. If $D_v$ contains an $S_h$, then since it consists of $3hk$-expanding edges, we would be done by greedily extending it into a copy of $\St_3(h,k),$ since $3hk \ge hk+h+1=|\St_3(h,k)|$. Hence, we suppose there is no $S_h$ in $D_v$, i.e.\ every vertex has degree at most $h$ in $D_v$. Take a maximal collection of disjoint $k$-stars in $D_v$. If this collection consists of at least $h$ stars, they together with $v$ make $\St_3(h,k)$ and we are done, so let us assume towards a contradiction that there are less than $h$ of them. So the union of the stars in this collection has at most $hk+h\le 2hk$ vertices, which participate in at most $2hk\cdot h\le e/(4n)$ edges of $D_v$, since we have shown that degree of any vertex is at most $h$. This means that among the remaining at least $e/(4n)\geq kn$ edges in $D_v$ we can find a new $S_k$, contradicting maximality of our collection.

So we may assume that $G$ contains at most $e/2$ edges containing a $3hk$-expanding pair; removing all such edges we obtain a graph $G'$ in which there is no $3hk$-expanding pair of vertices. We can find a vertex with degree at least $e/n$ and once again keep finding $k$-stars in its link graph for as long as we have less than $h$ of them. At any point we have at most $hk$ vertices and they have degree at most $3hk$ in the link graph of the found vertex (since we removed all edges containing a $3hk$-expanding pair), so they always touch at most $3h^2k^2\leq e/(4n)$ edges, hence we have at least $e/(4n)\geq kn$ other edges, and we can find a new $S_k$, as desired.
\end{proof}

\textbf{Remark.} One can show the bound in the proposition is tight, up to a constant factor, provided $2\le k \le n/h$. Indeed, \Cref{lem:S_3(2k)} gives the first term, and taking $n/(hk)$ disjoint copies of the complete $3$-graph on $hk$ vertices gives the second. Note also that, since $|\St_3(h,k)|>hk$, when $hk>n$ we can never find a copy of $\St_3(h,k)$ in a graph on $n$ vertices, hence the largest $\St_3(h,k)$-free graph is complete; this completes the picture on Tur\'an numbers of generalised $3$-uniform stars (case $k=1$ being \Cref{lem:3-uniform-star}).

The next result shows that if $|E(G)|\gg \max\{kn^2,h^2k^2n\}$ then not only $G$ contains a copy of $\St_3(h, k)$, as guaranteed by the above proposition, but it contains many disjoint copies of them.

\begin{thm}\label{thm:3-uniform-many-star-lastr}
    Every $3$-graph on $n$ vertices with $e\gg sn^2$ 
    edges, with $s \ge \max\{k,h^2k^2/n\}$, contains at least $t=\min\{s,\sqrt{sn}/h,s^{1/3}n^{2/3}/(hk)\}$ 
    vertex-disjoint copies of $\St_3(h, k)$.
\end{thm}

\begin{proof} 

Let $G$ be a $3$-graph on $n$ vertices with $e$ edges. Set $L_2=e/(3ht)$ and $L_3=e/(3hkt)$. We call a star $\St_3(h,k)$ in $G$ {\em well-behaved} if its $h$ second layer vertices have degrees at most $L_2$ and its $hk$ third layer vertices have degrees at most $L_3$. Suppose we have found a collection of less than $t$ vertex-disjoint well-behaved generalised stars $\St_3(h,k)$. The used vertices touch at most
$t\cdot\binom{n}{2}+ht\cdot L_2+hkt\cdot L_3\le \frac{3}{4}e$ edges. We call this set of edges $R$.

 Our goal now is to show that we can find a new well-behaved star $\St_3(h,k)$ in $G$ among the vertices which are not contained in any of the previous stars. Let $A$ denote the set of vertices with degree at least $L_2$, $B$ the set of vertices with degrees between $L_3$ and $L_2$ and $C$ the set of vertices with degree at most $L_3$. Note that $|A\cup B| \le 9ht+9hkt \le 18hkt$. Furthermore, there exists a subset $F\subseteq E(G)\setminus R$ such that every edge in $F$ contains the exact same number of vertices in $A$, as well as in $B$, and in $C$, and we have $|F|\ge \frac{1}{10}\cdot\frac{1}{4} e \gg sn^2$, since there are\footnote{The number of non-negative integer solutions to $x_1+x_2+x_3=3$.} $\binom{3+2}{2}=10$ different types of edges according to how many vertices they have in each of the sets $A,B$ and $C$. We distinguish four cases.
 
\begin{enumerate}[label={\em Case (\roman*).}]

    \item The edges in $F$ have no vertices in $C$.
    
We have $|F|\le \binom{|A\cup B|}{3}\lesssim (hkt)^3 \le sn^2 \ll |F|$, a contradiction.

\item All three vertices of each edge in $F$ are in $C$.

We can find a copy of $\St_3(h,k)$ using only the edges in $F$ by \Cref{prop:3-uniform-single-star-lastr}, since $s \ge \max\{k,h^2k^2/n\}$ and $|F| \gg sn^2$. As all its vertices are in $C$ the star is well-behaved.

\item All edges of $F$ have exactly two vertices in $C$.

There must exist a vertex in $A \cup B$ of degree at least $\frac{|F|}{|A\cup B|}\ge \frac{|F|}{18hkt}\gg \frac{sn^2}{hkt}$. Hence, we can use \Cref{lem:graph-many-stars} to find $h$ vertex-disjoint copies of $S_k$ in its link graph. To see why the lemma gives this, note that for $s'=sn/(hkt)$ we have $\min\{s', \sqrt{s'n}/k\} \ge h$ and $s' \ge k$. Indeed, we have $$\frac{s'}{h}=\frac{\sqrt{sn}}{hk}\cdot \frac{\sqrt{sn}}{ht} \ge 1,  \:\:\:\: \:\: \:\: \:\: \:\: \frac{\sqrt{s'n}}{hk}=t\cdot\sqrt{\frac{sn^2}{(hkt)^3}} \ge t \ge 1\:  \:\:\:\:\:\:\text{ and } \:\: \:\:\: \:\:\frac{s'}{k}=\frac{s}{k} \cdot \frac{s^{1/3}n^{2/3}}{hkt}\cdot \frac{n^{1/3}}{s^{1/3}} \ge 1,$$
where we used $s \ge h^2k^2/n$ and $t \le \sqrt{sn}/h$ in the first inequality, $t \le s^{1/3}n^{2/3}/hk$ in the second and last inequalities, where we also used $n \ge s\ge k$. 
This yields a copy of $\St_3(h,k)$ with all second and third layer vertices in $C$, so the copy is well-behaved.

\item Each edge in $F$ has exactly one vertex in $C$.

Let $\mathcal{P}$ be the set of $sn^2/(hkt)^2$-expanding pairs of vertices in $A\cup B$.
The total number of edges in $F$ which do not contain a pair in $\mathcal{P}$ is less than $\binom{|A\cup B|}{2}\cdot sn^2/(hkt)^2 \le 9 \cdot 18 sn^2 \le |F|/2$. The remaining, at least $|F|/2$ edges each contain a pair from $\mathcal{P}$ and each such pair can belong to at most $n$ edges, so $|\mathcal{P}| \ge |F|/(2n) \gg sn$.
Since at most $\binom{|A|}{2} \le \binom{9ht}{2}< 41sn$ pairs have both vertices in $A$, there are at least $18sn$ pairs in $\mathcal{P}$ with at least one vertex in $B$. This means that there is a vertex in $A\cup B$ having at least $ 18sn/|A\cup B| \ge sn/(hkt) \ge h$ neighbours (with respect to $\mathcal{P}$) in $B$, so we find a copy of $S_h$ in $\mathcal{P}$ with leaves in $B$. Since the pairs in $\mathcal{P}$ have degrees at least $sn^2/(hkt)^2 \ge hkt \ge hk$ (using $t \le s^{1/3}n^{2/3}/hk$), we can greedily extend this $S_h$ into a copy of $\St_3(h,k)$, which is well-behaved since every edge in $S_h$ had at most one vertex in $A$ and exactly two in $A \cup B$ so the third vertex of any edge containing it must be in $C$, by the case assumption.
\end{enumerate}

\vspace{-1.00cm}
\end{proof}

\textbf{Remark.} This result is again best possible in a number of ways. We need the bound on $s$ in order to be able to find even a single star, since \Cref{prop:3-uniform-single-star-lastr} is tight, as explained by the remark bellow it. The actual number of stars we find is also optimal: the bound $t \le s$ follows by taking $K_{s,n,n}$, $t \le \sqrt{sn}/h$ by taking $K_{\sqrt{sn},\sqrt{sn},n}$ (note that $\sqrt{sn} \le n$) and $t \le s^{1/3}n^{2/3}/(hk)$ by taking the complete graph on $s^{1/3}n^{2/3}\le n$ vertices.


The following lemmas arise as parts of our argument bounding $\ex(n,\St_4(d,d,k)$, where $d=(n/k)^{1/3}$, in the following subsection. They are either strengthenings of the above bounds or allow us more control over where we find vertices of our $3$-uniform stars.

\begin{lem}\label{last regime-bounded degree}
Let $H$ be a $3$-graph on $n$ vertices with $e \gg kn^2$ edges. If every pair of vertices belongs to less than $3kd^2$ edges of $H$, then $H$ contains $d$ vertex-disjoint copies of $\St_3(d,k)$, where $d=(n/k)^{1/3}$.
\end{lem}

\begin{proof}

We first show that for any vertex $w$ of degree at least $7kn$ in $H$, there is an $\St_3(d,k)$ centred at $w$. Observe first that the link graph $L_w$ does not contain a star of size $3kd^2$, as otherwise the centre of this star and $w$ make a pair of too high codegree. This means that in the link graph of $w$ we can find $d$ vertex-disjoint copies of $S_k$, again by a greedy procedure. Indeed in every step we have used at most $2kd$ vertices, each having degree at most $3kd^2$ in $L_w$, so in total they touch at most $6k^2d^3=6kn$ edges, and we can find the new $S_k$ among the remaining $nk$ edges. This gives us the desired $\St_3(d,k)$ centred at $w$.

Let $A$ be the set of vertices with degree larger than $ed/(8n) \gg nkd$ and $B$ the set of vertices with degree at most this. Since we have $e$ edges in total this implies $|A|\le 24n/d.$ This in turn implies that there can be at most $\binom{|A|}{2} \cdot 3kd^2 \le e/3$ edges with at least $2$ vertices in $A$, so we can remove them to get a subgraph $H'\subseteq H$ in which every vertex has at most one vertex in $A$ and $|E(H')|\ge 2e/3 \gg kn^2$.

Let us first assume there are at least $e/3$ edges with all vertices in $B$. Taking a maximal collection of vertex-disjoint $\St_3(d,k)$ using only such edges, either we are done (if we have found $d$ of them) or have used at most $2kd^2$ vertices, each with degree at most $ed/(8n)$. So in total, currently used vertices touch at most $2kd^2\cdot ed/(8n) = e/4$ edges of $H'$. Upon removing them we are left with at least $e/12 \gg kn^2$ edges with all vertices in $B$ and disjoint from the set of already used vertices. In particular, there is still a vertex of degree at least $7kn$ and we can use the observation from the beginning of the proof to find an additional $\St_3(d,k)$, a contradiction.

So we may assume there are at least $e/3$ edges with one vertex in $A$ and $2$ in $B$. 
Now take a maximal collection of vertex-disjoint copies of $\St_3(d,k)$ with centres in $A$ and remaining vertices in $B$. Either we are done or we have used at most $d$ vertices from $A$ and $2kd^2$ vertices in $B$. The former have degrees at most $n \cdot 3k d^2$ (second vertex we may choose in $n$ many ways, but for the final we are restricted by the codegree assumption) so touch at most $d \cdot n \cdot 3k d^2=3n^2 \le e/24$ edges. The latter have degrees at most $ed/(8n)$, by definition of $B$, so touch at most $2kd^2 \cdot ed/(8n)=e/4$ edges. Hence, upon removing all these edges we are left with at least $e/24 \ge 7kn^2$ of our edges, all of which are disjoint from the set of already used vertices. This means there is a vertex in $A$ with degree at least $7kn$, so once again using our initial observation we find an additional $\St_3(d,k)$, a contradiction.
\end{proof}

The following lemma allows us to find many copies of $3$-uniform stars with an added restriction that its leaves should avoid a relatively small subset of vertices, provided the edges of our graph have at least one vertex in this small subset. The choice of parameters might seem a bit arbitrary, but it arises from our intended application of the lemma in the $4$-uniform case.

\begin{lem}\label{lem:last regime-first claim}
Let $n,k$ and $t$ be positive integers with $t \le \min \{k,\sqrt{d}\},$ where $d=(n/k)^{1/3}$. Let $H$ be an $n$-vertex $3$-graph with $e \gg \frac{n^3}{td^2}$ edges. Let $C \subseteq V(H)$ such that $|C|\le {16nt}/d$ and assume that every edge of $H$ has precisely one vertex in $C$. Then $H$ contains $d$ vertex-disjoint copies of $\St_3(d,k)$ whose leaves lie outside of $C$.
\end{lem}
        \begin{proof}
        Let $D=V(H) \setminus C$. Suppose there exists a set $X$ of $d$ vertices in $D$ with degree at least \begin{equation}\label{eq:delta}
            \delta:=\frac{ed}{8n} \gg \frac{n^2}{dt}=\frac{nd^2k}t\ge \max(nd^2,ntdk)
        \end{equation}
        where in the last inequality we used the assumption on $t$. We delete from $H$ any edge containing two vertices in $X$. We deleted at most 
        $$d^2|C|\le 16ntd \ll  \frac{\delta}k$$ 
        such edges, using \eqref{eq:delta}. Hence, vertices in $X$ after deletion still have degree at least $\frac{\delta}2.$ We now show how to find $d$ vertex disjoint copies of $\St_3(d,k)$ with centres in $X$, 2nd layer vertices in $C$ and leaves in $D\setminus X$. We construct them as follows. We pick a vertex from $X$ and look inside its link graph for copies of stars $S_k$ with centres in $C$, leaves in $D \setminus X$ and not using any already used vertices. If we find $d$ of them we proceed to the next vertex of $X$, otherwise we stop. If we did not stop by the time we considered all vertices of $X$, we have found our desired structure.
        So we may assume that we do stop at some point when considering $x \in X$. At this point we have used at most $d^2$ vertices from $C$ and $kd^2+d\le 2kd^2$ vertices from $D$ (the first term being the contribution of $D \setminus X$ and the second of $X$). In the link graph $L_x$ (note that this graph consists of edges with one vertex in $C$ and one in $D$), already used vertices from $C$ touch at most $d^2\cdot n \le \delta/8$ 
        edges in total, using \eqref{eq:delta} and $t \le k$. Furthermore, used vertices from $D$ touch at most $$2kd^2|C| \le 32ntdk \le \frac{\delta}8$$ 
        edges in total, using \eqref{eq:delta}. So removing all edges touching these forbidden, already used, vertices we are left with at least 
        $$\frac{\delta}4 \ge \frac{16ntk}d\ge k|C|$$
        edges in $L_x$, avoiding all already used vertices. Hence, we find another appropriate $S_k$, a contradiction.
        
        So we may assume that there are at most $d$ vertices in $D$ with degree at least $\delta.$ These vertices participate in at most 
        $$dn^2 = \frac{n^3}{kd^2} \le \frac{e}4$$ 
        edges (using $t \le k$); we delete them and are left with a graph $H'$ with at least $\frac{3e}4$ edges such that all vertices in $D$ have degree at most $\delta$ in $H'$. Let us take a maximal collection of vertex-disjoint copies of $\St_3(d,k)$ with centres in $C$. If we are not done, the centres touch at most $dn^2 \le \frac{e}4$ edges, while the other used vertices touch at most $2kd^2 \cdot \delta = \frac{e}4$ edges (since they all belong to $D$ so touch at most $\delta$ edges). So there are at least $e/4$ edges which do not touch any of the used vertices. Let $H''$ be the graph consisting of these edges. Now it is enough to find another $\St_3(d,k)$ in $H''$, with its centre in $C$. Note that there is a vertex in $C$ of degree at least $$\frac{e}{4|C|}\ge \frac{6n^2}{dt^2}=6n\cdot \frac{kd^2}{t^2}$$ in $H''$. Applying \Cref{lem:graph-many-stars} with $s:=\frac{kd^2}{t^2}$, which we can since $s\ge kd \ge k$ (using $t \le \sqrt{d}$), we find $\min(s,\sqrt{sn}/k) \ge \min(kd,d^2) \ge d$ (using $n=kd^3$ and $t \le \sqrt{d}$ in the second term) disjoint stars $S_k$ in the link graph of this vertex, completing the proof.
        \end{proof}
     
     The final lemma for the $3$-uniform case is the following. It is similar in spirit to the above one, except that it works with smaller sets and only finds a single star. While the above lemma will be used to embed a number of stars within the link graph of a fixed vertex, the following one will be used to find the star making the first $3$-layers.
        
    \begin{lem}\label{lem:last regime-second claim}
   Let $n,k$ and $t$ be positive integers with $t \le \min \{k,\sqrt{d}\},$ where $d=(n/k)^{1/3}$. Let $H$ be a $3$-graph on vertex set $B \cup C$ where $B$ and $C$ are disjoint, $|B|\le 16td^2$ and $|C| \le 16td^2k$. If $H$ has $e \gg kn^{2}$ edges then it contains a copy of $\St_3(d,d)$ whose leaves all lie in $C$.
    \end{lem}
    \begin{proof}
    There are at most $$\binom{|B|}{3} \le 2^{12}t^3d^6\le 2^{12}k^3 (n/k)^2\le \frac{e}4$$ edges within $B$. If we have $e/4$ edges living completely in $C$ then by \Cref{prop:3-uniform-single-star-lastr}, since 
    $${e}/4 \ge 2^8n^2 \ge\max\{d|C|^2,d^4|C|\},$$
    where we used $t^2 \le d$ for the first term and $t \le k$ for the second, we can find a copy of $\St_3(d,d)$ consisting of these edges, so in particular having all leaves in $C$, as desired. So we may assume there are at least $e/2$ edges containing vertices from both $B$ and $C$. Let $H'$ be the subgraph consisting of such edges.
    
    A pair of vertices in $H$ is said to be a $B$-pair if it contains at least one vertex in $B$. There are at most $$|B| \cdot (|B|+|C|) \le 2^9t^2d^{4}k\le 2^9d^4k^3= 2^9kn^2/d^2 \le e/(12d^2)$$  
    $B$-pairs. We say a $B$-pair is \textit{$C$-expanding} if it extends into an edge of $H'$ using a vertex from $C$ in at least $3d^{2}$ many ways. This means there are at most $e/4$ edges which contain a non-$C$-expanding $B$-pair. The remaining, at least $e/4$, edges only contain $C$-expanding $B$-pairs. Each of these edges can be split into a $B$-pair, which must be $C$-expanding and a vertex in $C$. Hence, there needs to be at least
    $$\frac{e}{4|C|}\ge \frac{kn^2}{td^2k}= \frac{k}{t} \cdot \frac{n}{d^3k} \cdot dn \ge  dn$$ 
    $C$-expanding $B$ pairs. Hence, we can find a star of size $d$ made of such pairs and since each of them extends to at least $ 3d^{2}\ge |\St_3(d,d)|$ edges using a vertex in $C$, we find our desired star by a greedy procedure. 
    \end{proof}

\subsection{Optimal $4$-uniform unavoidable graphs} \label{subsec:-lower-bound}
We finally have all the $3$-uniform results our hearts might desire, we proceed to the $4$-uniform case. 

We begin by determining the Tur\'an number of a single copy of our unavoidable generalised $4$-uniform star. The proof is relatively simple since we did most of the legwork in the previous section.

\begin{lem}\label{lem:4-uniform-single-star-lastr}
For $1\le k \le n$ one has
$\ex(n,\St_4(d,d,k)) \lesssim kn^3$, where $d=(n/k)^{1/3}$.
\end{lem}

\begin{proof}
For ease of notation, set $\St=\St_4(d,d,k)$. Let $G$ be an $n$-vertex $4$-graph with $e \gg kn^3$ edges. We say a triple of vertices $X$ is \emph{expanding} in $G$ if it is $3kd^2$-expanding. 

\begin{claim*}
If at least $e/2 \gg kn^3$ edges contain an expanding triple, then $G$ contains $\St$.
\end{claim*}
\begin{proof}
Let $D_v:=\{X \subseteq V(G)^3 \mid X \cup v \in E(G),\: X \text{ is expanding}\}$. We have $\sum_v |D_v| \ge e/2,$ since every edge containing an expanding triple contributes at least one to the sum on the left. This implies that there exists a vertex $v$ with $|D_v| \ge e/(2n).$ Now, if there exists a vertex $w$ with degree at least $6dn$ in $D_v$ then we can find $d$ vertex-disjoint copies of the star of size $d$ within the link graph of $w$ in $D_v$, by \Cref{lem:graph-many-stars} and using that $n=d^3k \ge d^3$. Adding $w$ to these edges gives us an $\St_3(d,d)$ within $D_v$ (so consisting of expanding triples). This in turn implies we are done by greedily extending it to a copy of $\St$. So we may assume every vertex has degree at most $6dn$ in $D_v$. 

\Cref{prop:3-uniform-single-star-lastr} tells us that we can find at least one $\St_3(d,k)$ in any $3$-graph with $e/(4n) \gg kn^2$ edges (using that $n=d^3k \ge d^2k$ so that $d^2k^2n \le kn^2$). 
Suppose we have found $t<d$ vertex-disjoint copies of $\St_3(d,k)$ in $D_v$; they span at most 
$3kd^{2}$ vertices, each having degree at most $6dn$ in $D_v$. Hence, these vertices are incident to at most $18 n^2$ edges, so if we remove all of them, we are left with at least $e/(4n)$ edges. Now, \Cref{prop:3-uniform-single-star-lastr} implies that we can find another disjoint $\St_3(d,k).$ Therefore, we can find at least $d$ vertex-disjoint copies of $\St_3(d,k)$ inside $D_v$, which together with $v$ give us a copy of $\St$.
\end{proof}

By the claim we may assume there are at most $e/2$ edges containing an expanding triple, so after deleting them we are left with a subgraph $G'$ with at least $e/2 \gg kn^3$ edges, with no expanding triples. Take a vertex of maximum degree; it has degree at least $e/n \gg kn^2$ and in its link graph no pair of vertices belongs to more than $3kd^{2}$ edges since there is no expanding triple. So \Cref{last regime-bounded degree} implies the result.
\end{proof}

The following result gives us our optimal unavoidable $4$-graphs.

\begin{thm}\label{thm:last regime-lower bound}
In any $4$-graph $G$ on $n$ vertices with $e\gg kn^3$ edges, one can find $t=\min\{k,d^{1/4}\}$ vertex-disjoint copies of $\St_4(d,d,k),$ where $d=(n/k)^{1/3}$.
\end{thm}

\begin{proof}
Set $L_2=\frac{e}{4td},  L_3=\frac{e}{4td^2}$ and $L_4=\frac{e}{4td^2k}$. We partition $V(G)$ into level sets according to their degrees
\begin{align*}
    A&:=\left\{v \mid d(v) > L_2\right\}, \\
    B&:=\left\{v \mid L_3 < d(v) \le L_2\right\},  \\
    C&:=\left\{v \mid L_4 < d(v) \le L_3\right\} \enskip \text{and}  \\
    D&:=\left\{v \mid d(v) \le L_4\right\}.
\end{align*}
Let $\St:=\St_4(d,d,k)$. A generalised star $\St$ in $G$ is said to be \textit{well-behaved} if its second layer vertices are embedded among vertices with degree at most $L_2$ (i.e. vertices in $B\cup C\cup D$), the third layer vertices in vertices of degree at most $L_3$ (i.e. vertices in $C\cup D$), and the fourth layer vertices in vertices of degree at most $L_4$ (i.e. vertices in $D$).
Let us consider a maximal collection of vertex-disjoint well-behaved generalised stars $\St$. If we found less than $t$ stars then already used vertices touch at most 
$t\cdot \binom{n}{3}+td\cdot L_2+td^{2}\cdot L_3+ td^{2}k\cdot L_4 \le \frac78e$
edges. So there are at least $e/8$ edges disjoint from any previously used edges and our task is to to show we can embed an additional well-behaved $\St$ using these edges. 

Handshaking lemma gives us:
\begin{align*}
|A|& \le 4e/L_2=16td,\\ 
|B|& \le 4e/L_3=16td^{2} \enskip \text{and} \\ 
|C|& \le 4e/L_4=16td^{2}k.
\end{align*}
This implies $|A\cup B| \le 32 td^{2}$ and $|A\cup B\cup C| \le 48td^{2}k$.
Hence the total number of edges with at least $2$ vertices in $A$, or at least $3$ vertices in $A\cup B$, or all $4$ vertices in $A\cup B\cup C$
is upper bounded by
\begin{align*}
\binom{|A|}{2}\binom{n}{2} +\binom{|A\cup B|}{3}n+\binom{|A\cup B\cup C|}{4}  &\le 2^6t^2d^2n^2 &    &+2^{10}t^3d^6n& &+2^{19}t^4d^8k^4 & & \quad\quad\quad \\
 &\le 2^6d^{5/2}n^2 & &+2^{10}k^3 d^6 n & &+2^{19}d^9k^4& & \quad\quad\quad\\
 &\le 2^6n^3k   & &+2^{10}n^3k& &+2^{19}n^3k& &
\le 2^{20}kn^3 \le e/16, \quad\quad\quad
\end{align*}
where we used $t^4\le d$ and $t \le k$ in the second inequality and $n=kd^3$ in the third. In particular, there are at least $e/16$ remaining edges (which do not touch used vertices) such that they have at least $1$ vertex in $D$, at least $2$ vertices in $C \cup D$ and at least $3$ vertices in $B \cup C \cup D$. Furthermore, there exists a subset $F$ of these edges of size $|F|\ge (e/16)/ 35 \gg kn^3$ which all have the same number of vertices in each of $A, B, C$ and $D$, since there are\footnote{The number of non-negative integer solutions to $x_1+x_2+x_3+x_4=4$.} $\binom{4+3}{3}=35$ different types of edges according to how many vertices they have in each of the sets $A,B,C$ and $D$.
We distinguish several cases depending on how many vertices our edges in $F$ have in $D$.

\begin{enumerate}[label={\em Case (\roman*).}]

\item All edges in $F$ have $4$ vertices in $D$.

    Any $\St$ we find in this case is well-behaved so we are done by \Cref{lem:4-uniform-single-star-lastr}.

\item All edges in $F$ have exactly $3$ vertices in $D$.

There is a vertex $v$ in $A\cup B\cup C$ of degree at least $\frac{|F|}{|A\cup B\cup C|} \gg \frac{kn^3}{td^2k}=n^{2}dk/t$. Using \Cref{thm:3-uniform-many-star-lastr} (with $s:=dk/t$ and $h:=d$ so that $s \ge k$ and $ns=k^2d^4/t \ge h^2k^2$), we get at least $d$ vertex-disjoint copies of $\St_3(d,k)$, since $$s \ge d, \quad\quad \frac{\sqrt{sn}}h = \frac{kd}{\sqrt{t}}\ge d\quad\quad \text{ and } \quad\quad \frac{s^{1/3}n^{2/3}}{hk}=(d^4/t)^{1/3}\ge d.$$ Together with $v$, this gives a desired well-behaved copy of $\St$.
  
\item All edges in $F$ have exactly $2$ vertices in $D$.

    There can be at most $|A||B|\binom{n}{2} \le 128t^2n^3/k < |F|$ edges in $F$ with $1$ vertex in $A$ and $1$ in $B$. Hence our edges either have exactly $2$ vertices in $B$ and $2$ in $D$, or at least $1$ vertex in $C.$

\begin{enumerate}[label={\em Subcase (\alph*).}]

\item All edges in $F$ have $2$ vertices in $B$ and $2$ in $D$.\\
Given a vertex set $S$, a pair of vertices in $S$ is called an $S$-$S$ pair. Denote by $\mathcal{P}$ the set of $B$-$B$ pairs with at least $|F|/|B|^2 \ge 3nkd^2$ pairs in their link graph. Since there are $\binom{|B|}{2}$ different $B$-$B$ pairs, ones outside of $\mathcal{P}$ belong to at most $\binom{|B|}{2} \cdot |F|/|B|^2 \le |F|/2$ edges. Hence, the remaining $|F|/2$ edges have their $B$-$B$ pair belonging to $\mathcal{P}$. Each $B-B$ pair can extend into an edge in at most $\binom{|D|}{2}$ many ways. Hence, $$|\mathcal{P}| \ge \frac{|F|}{2}/\binom{|D|}{2} \ge 16kn \ge d|B|.$$ Therefore, $\mathcal{P}$ contains a star of size $d$. We embed vertex-disjoint stars $S_k$ into the link graphs of the leaves of our star in $\mathcal{P}$, dealing with one leaf at a time and moving to the next one when we found $d$ stars $S_k$ inside its link. Unless we are done, we have used at most $2kd^2$ vertices which can touch at most $2kd^2 \cdot n$ many $D$-$D$ pairs within the current link graph. Thus we still have at least $nkd^2 \ge kn$ many $D$-$D$ pairs in the link graph, disjoint from any already used vertices, which means we can find another $S_k,$ as desired. 
        
\item All edges have at least $1$ vertex in $C$ and exactly $2$ vertices in $D$.\\
        There is a vertex $v \in A\cup B\cup C$ which appears in at least $\frac{|F|}{|A\cup B\cup C|} \gg \frac{kn^3}{td^2k}= n^{3}/(td^2)$ edges with the remaining vertices being one in $C$ and two in $D$. By \Cref{lem:last regime-first claim} (note that $|C| \le 16td^2k=16nt/d$), the link graph of $v$ contains $d$ vertex-disjoint copies of $\St_3(d,k)$ whose leaves lie in $D$.
        
\end{enumerate}
 \item All edges in $F$ have exactly $1$ vertex in $D$.

Let $\mathcal{T}$ denote the collection of all triples in $A\cup B\cup C$ of codegree at least $$\frac{|F|}{|A\cup B\cup C|^3}\ge \frac{3kn^3}{(td^2k)^3}=\frac{3kd^3}{t^3}\ge 3kd^2.$$ Triples from $A\cup B\cup C$ outside $\mathcal{T}$ belong to at most $\binom{|A\cup B\cup C|}{3}\cdot \frac{|F|}{|A\cup B\cup C|^3}\leq|F|/2$ edges, so at least half of the edges in $F$ contain a triple in $\mathcal{T}$, and so $|\mathcal{T}| \ge |F|/(2n) \gg kn^2$. As there are at most $$|A|\cdot \binom{|A\cup B\cup C|}{2} \le 2^{16}t^3d^5k^2 \le 2^{16}d^6k^2=2^{16}n^2\le |\mathcal{T}|/2$$ triples with a vertex in $A$, we can remove them to obtain a collection $\mathcal{T'}$ of triples in $B \cup C$ such that $|\mathcal{T}'| \gg kn^2$. Now, if we find a well-behaved $\St_3(d,d)$ within $\mathcal{T}'$ (meaning that the second and third layer of vertices are embedded in $B\cup C$ and $C$ respectively) then we are done by greedily extending it and choosing distinct vertices in $D$ for $4$-th layer vertices of $\St$, which we can since each triple has codegree at least $3kd^2$. 
The existence of such $\St_3(d,d)$ is guaranteed by \Cref{lem:last regime-second claim}.
\end{enumerate}
\vspace{-1.1cm}
\end{proof}

Finally, the optimal unavoidable $4$-graphs at the very end of the range are of a very different flavour.

\begin{thm}\label{thm:truly last regime}
Every $n$-vertex $4$-graph with $e \ge n^{4-1/216}$ edges contains $\frac{1}{24}(e/n)^{1/4}$ vertex-disjoint copies of the complete $4$-partite $4$-graph $K_4(s,s,s,t)$, where $s=\frac{1}{12}\Big(\frac{\log n}{\log\left(n^4/e\right)}\Big)^{1/3}$ and $t=n^{1/4}$.
\end{thm}

As the proof is an easy consequence of the well-known K\"{o}v\'ari-S\'os-Tur\'an theorem and is very similar to its $3$-uniform analogue (\cite[Theorem 8]{chung1987unavoidable}), we defer it to the appendix.

\subsection{Lower bounds}
Let us now deduce the lower bound of the last regime of the unavoidability problem, i.e. we show the lower bound in \Cref{thm:main} (iii). Note first that if $e \le n^{4-\eps}$ for any $\eps>0$ we know $\frac{e^{1/4}\log n}{\log(\binom{n}{4}/e)} \approx e^{1/4}$. In particular, for $n^3 \ll e \le n^{4-1/216}$, we can use \Cref{thm:last regime-lower bound} with $k \gtrsim e/n^3$, to conclude there is an $(n,e)$-unavoidable $4$-graph with $$\min(k,(n/k)^{1/12}) \cdot (n/k)^{2/3} k= \min(n^{2/3}k^{4/3},n^{3/4}k^{1/4}) \gtrsim \min(e^{4/3}/n^{10/3}, e^{1/4})$$ edges, showing the desired bound. 

Similarly for $n^{4-1/216} \le e \ll n^{4}$, \Cref{thm:truly last regime} provides us with an $(n,e)$-unavoidable graph with $$\gtrsim (e/n)^{1/4} \frac{\log n}{\log (n^4/e)}n^{1/4} \gtrsim \frac{e^{1/4}\log n}{\log(\binom{n}{4}/e)}$$ edges, as desired.

\subsection{Upper bounds}\label{subsec:upper-bounds}
The results of the previous sections complete the picture in terms of lower bounds on $\un_4(n,e).$ Let us now turn to the upper bounds. They turn out to be much simpler than in the previous case, largely thanks to the following easy counting lemma from \cite{chung1987unavoidable}.

\begin{lem}\label{lem:general-ub}
If an $r$-graph $H$ on $p$ vertices with $q$ edges is $(n,e)$-unavoidable then 
$$q < \frac{p \log n}{\log \left(\binom{n}{r}/e \right)}. $$
\end{lem}

The following result, together with monotonicity of $\un_4(n,e)$ establishes the upper bounds for \Cref{thm:main} (iii) and completes its proof.

\begin{thm}\label{thm:4-ub and 5-ub}\textcolor{white}{ }
\begin{enumerate}
    \item[\rm (i)] For $n^3\ll e \ll n^{40/13}$ we have $\un_4(n,e) \lesssim e^{4/3}/n^{10/3}$.
    \item[\rm (ii)] For $n^{40/13} \ll e \le \binom{n}{4}$ we have $\un_4(n,e)\lesssim \frac{e^{1/4}\log n}{\log \left(\binom{n}{4}/e\right)}$.
\end{enumerate}
\end{thm}

\begin{proof}[Proof of \Cref{thm:4-ub and 5-ub}]
(i) Let $H$ be an $(n,e)$-unavoidable graph. Let $t=3(e/n)^{1/3}$. Take $n/t$ disjoint copies of a $4$-uniform clique on $t$ vertices. In total this gives us $\frac{n}{t}\binom{t}{4}>e$ edges so this graph must contain $H$ as a subgraph. In particular, every connected component of $H$ has size at most $t$. Now take another graph with a set $V_1$ of $50e/n^3 \le n/2$ vertices and a set $V_2$ of $n/2$ vertices and we pick all edges having one vertex in $V_1$ and three in $V_2$; this graph has more than $e$ edges, so it must contain $H$ as a subgraph. This implies $H$ can have at most $50e/n^3$ connected components. In turn this implies $H$ has at most $50te/n^3= 150e^{4/3}/n^{10/3}$ vertices. Now \Cref{lem:general-ub} implies $H$ has at most $\frac{13}{12} |V(H)|$ edges, which completes the proof. 

(ii) Let $H$ be an $(n,e)$-unavoidable graph. Take a $4$-uniform clique on $m=3e^{1/4}$ vertices; it has $\binom{m}{4}>e$ edges so must contain $H$,  hence $H$ has at most $m$ vertices. The claimed upper bound again follows from \Cref{lem:general-ub}.
\end{proof}



\section{Concluding remarks and open problems}\label{sec:conc-remarks}
In this paper we resolve a question of Chung and Erd\H{o}s which asks to determine the order of magnitude of $\un_4(n, e)$ defined as the maximum number of edges in a $4$-graph $G$ which is contained in every $4$-graph on $n$ vertices and $e \ll \binom{n}{4}$ edges. The most immediate open question is to answer their question for any uniformity.

\begin{qn}[Chung and Erd\H{o}s, 1983]\label{qn:1}
What is the order of magnitude of $\un_r(n,e)$ for any $r$?
\end{qn}

From our result the answer is now known for $r \le 4$ and $e \ll \binom{n}{4}$. In addition, our methods and certain further partial results give some indication about how the answer should behave for larger uniformities as well. For example, it seems likely that in general there are $\ceil{\frac{r}{2}}+1$ different regimes. The first one, when $e \ll n^{\floor{\frac{r}{2}}}$ always has an easy answer of $\un_r(n,e)=1$ and the following regimes are $n^{i}\ll e \ll n^{i+1}$ where $\floor{\frac{r}{2}} \le i \le r-1.$ In each regime (with the exception of 
$i=(r-1)/2$ and $i=r-1$) there are two competing bounds, which arise from the fact there are two 
sunflowers $\Su_r(t,2)$ with $\ex(n,\Su_r(t,2))\approx n^i$. Our answer in the second regime for the $4$-uniform case turned out to be a bit surprising and is in fact in-between the two natural guesses, so we are not willing to conjecture the correct value of the turning point for general uniformity. In the last regime, $i=r-1,$ the upper bound given by \Cref{thm:4-ub and 5-ub} generalises easily and seems to give the correct answer for all uniformities (as long as $e \ll n^r$).
\begin{conj}\label{conj:1}
For any $r \ge 2$ and $\eps>0$, provided $n^{r-1} \ll e \ll n^{r-\eps}$ we have 
$$\un_r(n,e)\approx \min(e^{\frac{r}{r-1}}/n^{r-1+\frac1{r-1}},e^{\frac1r}).$$
\end{conj}
Here, even the optimal unavoidable $r$-graphs seem to be clear, namely they should consist of an appropriate number of copies of $\St_r((n/k)^{1/(r-1)},\ldots,(n/k)^{1/(r-1)}, k),$ however bounding their Tur\'an numbers seems to be highly non-trivial. The assumption $e \ll n^{r-\eps}$ was made since it is not hard to generalise \Cref{thm:truly last regime} (we do so in the Appendix) and hence determine $\un_r(n,e)$ for all $n^{r-\eps}\ll e \ll n^r$ for some $\eps>0$.

The main stumbling block for extending our methods to higher uniformities is the fact that Tur\'an numbers of sunflowers $\ex(n,\Su_r(t,k))$ are not very well understood when $r\ge 5$ and $k$ is allowed to depend on $n$, as pointed out by Chung and Erd\H{o}s in \cite{chung1987unavoidable}. The main reason being that these sunflowers represent main building blocks for all our examples, across most of the range. On this front, the appropriate generalisation of \Cref{thm:4-uniform-sunflowers} seems to be as follows.

\begin{conj} For every fixed $r \ge 5$ and $ t <r$ one has 
$\ex(n,\Su_r(t,k))\approx k^{\min\{t+1,r-t\}} n^{\max\{r-t-1,t\}}$.
\end{conj}
This would generalise a result of Frankl and F{\"u}redi \cite{FF85} and F{\"u}redi \cite{Furedi83} (who solve it when $k$ is a constant) and the question may be attributed to Chung and Erd\H{o}s. One can generalise our constructions from \Cref{sec:sunflower} to show the lower bound part, and some methods for upper bounds also generalise. We can prove this conjecture for several more uniformities, although even in the case $r=5$ we needed additional ideas.

Since generalised stars seem to be optimal unavoidable graphs, as long as $e\ll n^{r-1}$, (at which point their unions take over) the following seems to be the key problem one needs to resolve in order to answer \Cref{qn:1}.
\begin{qn}\label{qn:gen-star}
Let $r$ be fixed. Determine the order of magnitude of $\ex(n, \St_r(d_1,\ldots, d_{r-1}))$, where $d_i$'s are allowed to depend on $n$.
\end{qn}
One can read out the answer for $r=3$ from \Cref{prop:3-uniform-single-star-lastr} and we believe our methods suffice to also solve it for $r=4$. Yet for higher uniformities even the case in which we keep the $d_i$'s fixed, which is yet another generalisation of the result of Frankl and F{\"u}redi \cite{FF85} and F{\"u}redi \cite{Furedi83} on Tur\'an numbers of sunflowers with fixed uniformity and number of petals, seems potentially interesting.

Both \Cref{conj:1} and \Cref{qn:gen-star} are examples of an interesting general question. Tur\'an numbers of both graphs and hypergraphs are well-studied, but in most cases one is only interested in Tur\'an numbers of a graph of fixed size. For many classical examples one can ask what happens if the fixed size restriction is removed. This can be very useful in a number of situations, perhaps the most ubiquitous being the K\"{o}v\'ari-S\'os-Tur\'an theorem \cite{Kovari-Sos-Turan} which is often used to find complete bipartite graphs of order even comparable to that of the underlying graph. For some additional examples see \cite{erdos-book,Furedi15,alon-sparse-turan}. It definitely seems there is plenty of potential for interesting future work in this direction.

So far we have avoided discussing the assumption $e \ll n^4$ in \Cref{thm:main}, mostly following in line of Chung and Erd\H{o}s. In fact we can replace this condition in \Cref{thm:main} with $e\le \binom{n}{4}-n^{1+c}$ for any $c>0$ (it requires choosing $s$ and $t$ slightly differently in \Cref{thm:truly last regime}). The problem seems to change significantly at this point and attains a very different flavour. Even the graph case, which was raised by Chung and Erd\H{o}s in 1983 has only recently been resolved in \cite{BDS19} and it suggests that around this point the optimal extremal examples seem to become (pseudo)random graphs in place of the complete $r$-partite graphs and the answer changes. Given that even the graph case turned out to be somewhat involved and relies on completely different ideas, we leave this open for future research.

\begin{qn}
For $r \ge 3$ determine the order of magnitude of $\un_r(n,e)$ when $e=(1-o(1))\binom{n}{r}$.
\end{qn}

Another natural follow-up question is to determine how optimal (up to a constant factor) $(n,e)$-unavoidable $r$-graphs look like. It is entirely possible to answer this question without answering \Cref{qn:1}, since one can potentially force the structure of an optimal unavoidable graph similarly as in \Cref{thm:middle-upper-bound} (as was demonstrated by Chung and Erd\H{o}s in the $3$-uniform case). At the very least we believe that optimal unavoidable graphs should be (close to) generalised stars as long as $e$ is not too close to $\binom{n}{4}$. Towards the end of the range the situation becomes blurry, as at least for part of the range, both copies of complete $r$-partite graphs and generalised stars are simultaneously optimal.

\providecommand{\bysame}{\leavevmode\hbox to3em{\hrulefill}\thinspace}
\providecommand{\MR}[1]{}
\providecommand{\MRhref}[2]{%
  \href{http://www.ams.org/mathscinet-getitem?mr=#1}{#2}
}
\providecommand{\href}[2]{#2}

\vspace{-0.1cm}
\appendix

\section{Large complete $r$-partite subgraphs of dense $r$-graphs}  

In this section we provide a proof of (a generalisation of) \Cref{thm:truly last regime}.

The proof is based on K\"{o}v\'ari-S\'os-Tur\'an theorem \cite{Kovari-Sos-Turan}, for hypergraphs. Since the sizes of the $r$-partite graphs we want to find grow with the number of vertices we need to go through the standard proof with care. The starting point is the graph case.

\begin{thm}[K\"{o}v\'ari, S\'os and Tur\'an \cite{Kovari-Sos-Turan}]
\label{thm:Kovari-Sos-Turan}
If $G=(A\cup B,E)$ is a bipartite graph and for some integers $s$ and $t$ we have
\[
t\binom{|A|}{s} < |B|\binom{|E|/|B|}{s}
\]
then $G$ contains a complete bipartite graph $K_{s,t}$ with the vertex class of size $s$ embedded in $A$.
\end{thm}

We will apply the following in the not so dense case.

\begin{lem}\label{lem:Kovari-Sos-Turan r-partite}
Let $n,r$ and $s$ be positive integers with $r\ge 2$ and $n\ge 2^{(4s)^{r-1}}$. Let $V_1,\ldots,V_r$ be sets of size $n$ and $M$ be a subset of $V_1\times \cdots \times V_r$ with $|M|\ge n^{r-1/(4s)^{r-1}}$. Then there exists $A_1\times \cdots \times A_r \subset M$ such that $|A_i|=s$ for every $1\le i<r$ and $|A_r|=\sqrt{n}$. 
\end{lem}
\begin{proof}
We prove \Cref{lem:Kovari-Sos-Turan r-partite} by induction on $r$. The base case $r=2$ follows from \Cref{thm:Kovari-Sos-Turan} since
\[
\sqrt{n}\binom{n}{s} <n \binom{n^{1-1/(4s)}}{s}.
\]
Here we used the fact that $\binom{n}{s}/\binom{m}{s} \le (\frac{n}{m-s})^{s} \le (\frac{2n}{m})^s$ provided $m \ge 2s$.

Assume that the statement holds for $r-1$. Consider the bipartite graph $G=(A\cup B,M)$ in which $A=V_1$ and $B=V_2\times \cdots \times V_r$. Since \[
n^{r-1-1/(4s)^{r-2}}\binom{n}{s} <n^{r-1} \binom{n^{1-1/(4s)^{r-1}}}{s}
\]
for $n\ge 2^{(4s)^{r-1}}$, there exists $A_1\times M' \subset M$ with $|A_1|=s$ and $|M'|=n^{r-1-1/(4s)^{r-2}}$. By appealing to the induction hypothesis, we conclude $M'$ contains $A_2\times \ldots \times A_r$ such that
$|A_i|=s$ for every $2\le i \le r-1$ and $|A_r|=\sqrt{n}$. This completes the proof of \Cref{lem:Kovari-Sos-Turan r-partite}.
\end{proof}

We now show a similar bound for a number of copies of $r$-partite graphs.

\begin{thm}\label{thm:truly last regime - generalisation}
Every $n$-vertex $r$-graph $G$ with $e \ge n^{r-1/6^{r-1}}$ edges contains $\frac{1}{6r}(e/n)^{1/r}$ vertex-disjoint copies of the complete $r$-partite $r$-graph $K_r(s,\ldots,s,t)$, where $s=\frac{1}{12}\Big(\frac{\log n}{\log\left(n^r/e\right)}\Big)^{1/(r-1)}$ and $t=n^{1/r}$.
\end{thm}

\begin{proof}
A vertex of $G$ is called {\em expanding} if its degree is at least $re^{1-1/r}$. By the handshaking lemma, there are at most $e^{1/r}$ expanding vertices. Let $k$ denote the maximum number of vertex-disjoint copies of $K_r(s,\ldots,s,t)$ that can be embedded in $G$, whereas in each copy the $r$-th vertex class consists of $t$ non-expanding vertices. Suppose to the contrary that $k<(e/n)^{1/r}$. 

The number of edges containing some used vertices is at most $k(r-1)s\cdot \binom{n}{r-1}+kt\cdot r e^{1-1/r} \le e/3$ assuming $e \ge n^{r-1/r}$. Moreover, the number of edges within the set of expanding vertices is at most $\binom{e^{1/r}}{r} \le e/6$. Therefore, by removing those edges we obtain a subhypergraph $H$ with $e/2$ edges such that any edge of $H$ contains at least one non-expanding vertices, and the edges of $H$ don't touch used vertices. Let $V_1,\ldots,V_{r-1}$ be $r-1$ copies of $V(G)$, and  $V_r$ be a copy of the set of non-expanding vertices. Denote by $M$ the set of $r$-tuples $(v_1,\ldots,v_r)$ in $V_1\times \cdots \times V_r$ such that $\{v_1,\ldots,v_r\}$ is an edge of $H$. Clearly, $|M| \ge |E(H)| \ge e/2 \ge n^{r-1/(4s)^{r-1}}$.\footnote{Notice that $e/2 \ge n^{r-1/(4s)^{r-1}}$ provided $s\le \frac{1}{6}\Big(\frac{\log n}{\log (n^r/e)}\Big)^{1/(r-1)}$, while $n\ge 2^{(4s)^{r-1}}$ for $s \le \frac{1}{4} (\log n)^{1/(r-1)}$. Furthermore, for $e\ge n^{1-1/6^{r-1}}$ we have $\frac{1}{6}\Big(\frac{\log n}{\log (n^r/e)}\Big)^{1/(r-1)}\ge 1$.} \Cref{lem:Kovari-Sos-Turan r-partite} implies that there exists a set $A_1\times \cdots \times A_r \subset M$ satisfying $|A_1|=\ldots=|A_{r-1}|=s$ and $|A_r|=\sqrt{n} \ge t$.
The sets $A_1,\ldots,A_r$ are disjoint, for the edges of $H$ consist of distinct vertices. Hence $H$ contains a copy of $K_r(s,\ldots,s,t)$ in which the $r$-th vertex class is embedded in the set of non-expanding vertices, a contradiction to the maximality of $k$.
\end{proof}

\end{document}